\documentclass[a4paper,11pt]{article}
\usepackage{pdfsync}

\usepackage[plainpages=false]{hyperref}
\usepackage{amsfonts,amsmath,amssymb,amsthm,enumerate}
\usepackage{latexsym,lscape,rawfonts,mathrsfs,mathtools}
\usepackage{enumitem,pifont}
\usepackage{cases}

\usepackage[all]{xy}
\usepackage{eufrak}
\usepackage{makeidx}
\usepackage{graphicx,psfrag}
\usepackage{pstool}
\usepackage{tikz-cd}
\usepackage{float}
\usepackage{array}
\usepackage{diagbox}

\usepackage{setspace}

\usepackage[titletoc,title]{appendix}

\usepackage{txfonts}

\newcommand{\C}{\ensuremath{\mathbb{C}}}

\newcommand{\IF}{\ensuremath{\mathbb{F}}}

\newcommand{\MM}{\ensuremath{\mathcal{M}}}
\newcommand{\IN}{\ensuremath{\mathbb{N}}}

\newcommand{\R}{\ensuremath{\mathbb{R}}}
\newcommand{\RR}{\ensuremath{\mathcal{R}}}
\newcommand{\MS}{\ensuremath{\mathcal{S}}}

\newcommand{\Z}{\ensuremath{\mathbb{Z}}}
\newcommand{\CP}{\ensuremath{\mathbb{P}}}
\newcommand{\RP}{\ensuremath{\mathbb{RP}}}

\newcommand{\ba}{\begin{align*}}
\newcommand{\ea}{\end{align*}}
\newcommand{\na}{\nabla}
\newcommand{\la}{\langle}
\newcommand{\ra}{\rangle}
\newcommand{\lc}{\left(}
\newcommand{\rc}{\right)}

\newcommand{\ep}{\epsilon}

\newcommand{\ka}{K\"ahler\,}

\newcommand{\Rm}{\ensuremath{\mathrm{Rm}}}
\newcommand{\Rc}{\ensuremath{\mathrm{Rc}}}

\newcommand{\KN}{\mathbin{\bigcirc\mspace{-14mu}\wedge\mspace{3mu}}}
\newcommand{\XX}{\mathcal{X}}
\newcommand{\tf}{\mathfrak{t}}
\renewcommand{\t}{\mathfrak{t}}
\newcommand{\Var}{\text{Var}}
\newcommand{\CF}{\mathfrak{C}}
\newcommand{\II}{\mathcal{I}}

 \newcommand{\xmark}{\ding{53}}

\makeatletter
\newcommand*\owedge{\mathpalette\@owedge\relax}
\newcommand*\@owedge[1]{%
\mathbin{%
\ooalign{%
$#1\m@th\bigcirc$\cr
\hidewidth$#1\m@th\wedge$\hidewidth\cr
}%
}%
}
\makeatother

\makeatletter
\def\ExtendSymbol#1#2#3#4#5{\ext@arrow 0099{\arrowfill@#1#2#3}{#4}{#5}}

\makeatother

\makeatletter
\def\ExtendSymbol#1#2#3#4#5{\ext@arrow 0099{\arrowfill@#1#2#3}{#4}{#5}}
\newcommand\longright[2][]{\ExtendSymbol{-}{-}{\rightarrow}{#1}{#2}}
\makeatother

\def\Xint#1{\mathchoice
{\XXint\displaystyle\textstyle{#1}}%
{\XXint\textstyle\scriptstyle{#1}}%
{\XXint\scriptstyle\scriptscriptstyle{#1}}%
{\XXint\scriptscriptstyle\scriptscriptstyle{#1}}%
\!\int}
\def\XXint#1#2#3{{\setbox0=\hbox{$#1{#2#3}{\int}$ }
\vcenter{\hbox{$#2#3$ }}\kern-.55\wd0}}

\def\aint{\Xint-}

\numberwithin{equation}{section}

\newtheorem{thm}{Theorem}[section]

\newtheorem{prop}[thm]{Proposition}
\newtheorem{lem}[thm]{Lemma}

\newtheorem{rem}[thm]{Remark}

\newtheorem{defn}[thm]{Definition}

\newtheorem{exmp}[thm]{Example}

\title{On K\"ahler Ricci shrinker surfaces}
\author{Yu Li \quad and \quad Bing Wang}
\date{\today}

\begin{document}
\maketitle

\begin{abstract}
In this paper, we prove that any K\"ahler Ricci shrinker surface has bounded sectional curvature. By combining this estimate with previous research by numerous authors, we achieve a complete classification of all K\"ahler Ricci shrinker surfaces.
\end{abstract}

\tableofcontents

\section{Introduction}

A Ricci shrinker $(M^n, g, f)$ is a complete Riemannian manifold $(M^n,g)$ coupled with a smooth function $f$ satisfying
\begin{align}
\Rc+\text{Hess}\,f=\frac{1}{2}g \label{E100},
\end{align}
where the potential function $f$ is normalized so that
\begin{align}
R+|\nabla f|^2&=f \label{E101}.
\end{align}
Ricci shrinkers play a fundamental role in the study of singularities of the Ricci flow. For the type-I Ricci flow, it was proved by Enders-M\"uller-Topping \cite{EMT11} that any proper blowup sequence converges smoothly to a non-trivial Ricci shrinker. In general, Ricci shrinkers provide crucial insights into the ancient solutions that arise from blowup sequences of Ricci flows.

In the case that $\dim M=n \leq 3$, the classification of all Ricci shrinkers has been fully established (cf.~\cite{Ha95}\cite{Naber}\cite{NW}\cite{CCZ}, etc.). The complete list consists of $\mathbb{R}^2$, $S^2$, $\mathbb{R}^3$, $S^3$, $S^2\times \mathbb{R}$, as well as their corresponding quotients. Consequently, the natural focus of study shifts to the classification of Ricci shrinkers when $n=4$. Significant advancements have been made in this direction over the past decade. The majority of classification results, to the best of our knowledge, require certain assumptions on the curvature (cf. \cite{CC13}\cite{CW15}\cite{KW15}\cite{KW20}\cite{LN20}\cite{LNW18}\cite{MW17}\cite{Naber}\cite{Na19}\cite{N05}\cite{PW10}, etc.). In this paper, we drop all curvature assumptions in the K\"ahler category.

A Ricci shrinker $(M, g, f)$ is called a K\"ahler Ricci shrinker if it admits a K\"ahler structure $J$ such that $ \nabla f$ is a holomorphic vector field.
In complex coordinate charts, this means
\begin{align*}
R_{i\bar{j}} + f_{i\bar{j}}=\frac12 g_{i\bar{j}}, \quad f_{ij}=f_{\bar{i}\bar{j}}=0.
\end{align*}
K\"ahler manifolds of complex dimension $2$ are often referred to as K\"ahler surfaces. This paper focuses on K\"ahler Ricci shrinker surfaces $(M^2, g, J, f)$, where $2$ denotes the complex dimension (by a slight abuse of notation).

Assuming bounded curvature, a classification of such surfaces was previously obtained by Bamler-Cifarelli-Conlon-Deruelle~\cite{BCCD22}, building upon the foundational work of Munteanu-Wang~\cite{MW19} and Cifarelli~\cite{Cifarelli22}. However, it remains unclear whether the bounded-curvature assumption is necessary.  

In this paper, we remove the bounded-curvature assumption and establish a complete classification of all K\"ahler Ricci shrinker surfaces. (See also Theorem B of~\cite{BCCD22} for comparison.)

\begin{thm}[\textbf{Complete classification of K\"ahler Ricci shrinker surfaces}]
Let $(M^2,g,J,f)$ be a K\"ahler Ricci shrinker surface. Then $(M^2,g,J,f)$ is biholomorphic-isometric to one of the following manifolds.
\begin{enumerate}[label=\textnormal{(\roman{*})}]
\item A closed Fano surface with its unique K\"ahler Ricci shrinker metric;
\item The Gaussian soliton $(\C^2,g_E)$;
\item The FIK shrinker, constructed by Feldman-Ilmanen-Knopf \emph{\cite{FIK03}} on the blowup of $\C^2$;
\item The standard $(\CP^1 \times \C,g_c)$;
\item The BCCD shrinker, abstractly constructed by Bamler-Cifarelli-Conlon-Deruelle \emph{\cite{BCCD22}} on the blowup of $\CP^1 \times \C$.
\end{enumerate}
\label{thm:classification}
\end{thm}

Let us briefly discuss the development and background of this classification.
We consider two cases for the complete manifold $M$: closed and non-compact. We will discuss each case separately.

(I) \textit{$M$ is closed.}

Since $M$ is closed, we have $c_1(M,J)=\frac{1}{2\pi} [\Rc]>0$, which implies that $(M, J)$ is a del Pezzo surface—a smooth Fano variety of dimension two. The classification of del Pezzo surfaces was established a long time ago (cf. \cite{Nagata}\cite{Manin}\cite{Hit75}). According to the classification, $(M,J)$ must be biholomorphic to either $\CP^1 \times \CP^1$ or $\text{Bl}_{k} \CP^2$, which represents the blowup of $k$ points on $\CP^2$ in general position, with $0 \leq k \leq 8$. Here, ``general position'' means that no three points are collinear, no six points lie on the same conic, and no eight points lie on a cubic curve with a double point at one of them. For $\text{Bl}_{k} \CP^2$ with $5 \leq k \leq 8$, the configuration of these $k$ points, modulo the $\mathrm{PSL}(3, \C)$ action, forms a moduli space of positive dimension. Consequently, infinitely many Fano complex structures exist on $\text{Bl}_{k} \CP^2$.

\begin{table}[H]
\centering
\vspace*{12mm}
\begin{tabular}{|c|c|c|c|c|c|}
\hline
\diagbox[dir=NW]{\rule{0mm}{10mm}Topology}{Properties} & Compact &Trivial & Toric & Moduli dimension \\ \hline
$\CP^2$ & $\checkmark$ & \checkmark &$\checkmark$ &0 \\ \hline
$\CP^1 \times \CP^1$ & $\checkmark$ & \checkmark& $\checkmark$ &0 \\ \hline
$\text{Bl}_{1} \CP^2$ & $\checkmark$ & \xmark & $\checkmark$ &0 \\ \hline
$\text{Bl}_{2} \CP^2$ & $\checkmark$ & \xmark & $\checkmark$ &0\\ \hline
$\text{Bl}_{3} \CP^2$ & \checkmark & $\checkmark$ &\checkmark&0 \\ \hline
$\text{Bl}_{4} \CP^2$ & $\checkmark$ & \checkmark & \xmark &0 \\ \hline
$\text{Bl}_{5} \CP^2$ & $\checkmark$ & \checkmark & \xmark &2 \\ \hline
$\text{Bl}_{6} \CP^2$& $\checkmark$ & \checkmark & \xmark &4 \\ \hline
$\text{Bl}_{7} \CP^2$ & $\checkmark$ & \checkmark & \xmark &6 \\ \hline
$\text{Bl}_{8} \CP^2$ & $\checkmark$ & \checkmark & \xmark &8 \\ \hline
$\C^2$ & \xmark & \xmark & \checkmark &0 \\ \hline
$\CP^1 \times \mathbb{C}$ & \xmark & \xmark & \checkmark &0 \\ \hline
$\text{Bl}_{1} \C^2$ & \xmark & \xmark & \checkmark &0 \\ \hline
$\text{Bl}_{1} (\CP^1 \times \C )$ & \xmark & \xmark & \checkmark &0 \\ \hline
\end{tabular}
\caption{Complete classification of K\"ahler Ricci shrinker surfaces}
\end{table}

On a closed K\"ahler Ricci shrinker surface $M$, the potential function $f$ can either be a constant or non-constant. If $f$ is constant, then $g$ represents a Fano K\"ahler-Einstein metric, i.e., a trivial K\"ahler Ricci shrinker metric. The study of K\"ahler-Einstein metrics on del Pezzo surfaces was extensively conducted in the 1980s-1990s (cf. \cite{TY87}\cite{Siu88}\cite{Tian90}). It was discovered that every del Pezzo surface $(M, J)$ admits a K\"ahler-Einstein metric, except when it is biholomorphic to $\text{Bl}_{1}\CP^2$ or $\text{Bl}_{2}\CP^2=\text{Bl}_{1} (\CP^1 \times \CP^1)$. Moreover, the K\"ahler-Einstein metric is unique up to the action of $Aut(M,J)$, as demonstrated by the work of Bando-Mabuchi~\cite{BandoMabuchi}. This uniqueness theorem has been generalized to the K\"ahler Ricci shrinker metric by Tian-Zhu \cite{TZ00}\cite{TZ02}.
Clearly, if one can establish the existence of K\"ahler Ricci shrinker metrics on $\text{Bl}_{1}\CP^2$ and $\text{Bl}_{1} (\CP^1 \times \CP^1)$, one can conclude that each del Pezzo surface admits a unique K\"ahler Ricci shrinker metric.
In fact, such shrinker metrics do exist. On $\text{Bl}_{1} \CP^2$, the Koiso-Cao K\"ahler Ricci shrinker metric can be found by solving an ordinary differential equation using Calabi's ansatz (cf. \cite{Koiso90}\cite{Cao96}).
On $\text{Bl}_{2} \CP^2$, the existence of such a metric was established by Wang-Zhu \cite{WZ04} through solving a Monge-Amp\`ere equation on a toric manifold.

In sum, each del Pezzo surface admits a unique K\"ahler Ricci shrinker metric.

(II) \textit{$M$ is noncompact.}

The classification of K\"ahler Ricci shrinker surfaces is also known when the scalar curvature of $M$ is bounded, i.e.,
\begin{align}
\sup_{M} R<+\infty. \label{eqn:sbound}
\end{align}
In fact, the curvature estimate established by Munteanu and Wang~\cite[Theorem 1.1]{MW15a} implies that under the condition \eqref{eqn:sbound}, $M$ has bounded Riemannian curvature:
\begin{align}
\sup_{M} |\Rm|<+\infty. \label{eqn:rmbound}
\end{align}
In view of (\ref{eqn:sbound}), the potential function $f$ does not have critical points outside a compact set, and the flow lines of $\nabla f$ can be parametrized by the level sets of $f$. Consequently, two possibilities arise: either the scalar curvature tends to zero along each flow line, or there exists a flow line along which the scalar curvature approaches a positive number.

By combining the curvature estimate (\ref{eqn:rmbound}) with the K\"ahler condition, Cifarelli, Conlon, and Deruelle \cite{CCD22} have demonstrated that the aforementioned dichotomy can be expressed more simply as: either $\displaystyle \lim_{x \to \infty} R(x)=0$ or $\displaystyle \lim_{x \to \infty} R(x)=1$.

$\bullet$ If $\displaystyle \lim_{x \to \infty} R(x)=0$, then $M$ is the Gaussian metric on $\C^2$ or the FIK metric on $\text{Bl}_{1} \C^2$. 

This result is established by Conlon, Deruelle, and Sun (cf. \cite{CDS19}). Here, we provide a brief outline of the key points for the readers' convenience. By applying the result of Munteanu and Wang \cite[Theorem 5.1]{MW15a}, it is deduced that in this case, the Riemannian curvature decays quadratically, and $M$ approaches an asymptotic metric cone with a smooth link. In Conlon, Deruelle, and Sun's work (cf. \cite[Theorem E (3)]{CDS19}), rigidity from complex analytic geometry is employed to demonstrate that the limiting cone must be biholomorphic to $\C^2$. Thus, by the adjunction formula, it follows that $M$ is biholomorphic to either $\C^2$ or $\text{Bl}_{1} \C^2$. In both cases, they establish that $g$ possesses $U(2)$-symmetry and prove the uniqueness of the shrinker metric through the work of Feldman, Ilmanen, and Knopf \cite{FIK03}. In the former case, $g$ is isometric to the Gaussian shrinker metric, while in the latter case, $g$ is isometric to the Feldman-Ilmanen-Knopf metric (cf. \cite{FIK03}).

$\bullet$ If $\displaystyle \lim_{x \to \infty} R(x)=1$, then $M$ is the standard $\CP^1 \times \C$, or the BCCD metric on $\text{Bl}_{1} (\CP^1 \times \C )$.

The classification of K\"ahler Ricci shrinker surfaces with $\displaystyle \lim_{x \to \infty} R(x)=1$ is completed in the works of Cifarelli-Conlon-Deruelle~\cite{CCD22} and Bamler-Cifarelli-Conlon-Deruelle~\cite{BCCD22}. In Cifarelli-Conlon-Deruelle \cite{CCD22}, it is established that $M$ is biholomorphic to either $\CP^1 \times \C$ or $\text{Bl}_{1} (\CP^1 \times \C )$. Moreover, in the former case, they prove that the metric must be isometric to the standard metric on $\CP^1 \times \C$. However, in the latter case, while uniqueness is established, the existence of the metric is still conjectural (cf. Conjecture 1.1 of \cite{CCD22}). In \cite{BCCD22}, Bamler, Cifarelli, Conlon, Deruelle study the blowup limit of a K\"ahler Ricci flow with toric symmetry, which corresponds to a non-compact K\"ahler Ricci shrinker. Utilizing divisor rigidity and toric symmetry (cf. \cite{CLW08}\cite{CW12}), it is determined that the infinity of this shrinker is $\CP^1 \times \C$, and it contains an exceptional curve. Therefore, it must correspond to the conjectural K\"ahler Ricci shrinker on $\text{Bl}_{1} (\CP^1 \times \C )$. It is an interesting problem to  find also a traditional construction by solving PDE, similar to the cases $\text{Bl}_{1} (\CP^2)$ and $\text{Bl}_{2} (\CP^2)$. 

From the preceding discussion, it is evident that the only remaining component needed to complete the classification is the scalar curvature estimate (\ref{eqn:sbound}) on each K\"ahler Ricci shrinker surface. This estimate is accomplished by the following theorem.

\begin{thm}[\textbf{Main result}]
\label{T102}
Any K\"ahler Ricci shrinker surface has bounded scalar curvature and consequently bounded sectional curvature. 
\end{thm}

Bounding the scalar curvature is a significant problem in K\"ahler geometry. One notable contribution is made by G. Perelman, who established a uniform scalar curvature bound along the K\"ahler Ricci flow on Fano manifolds (cf. \cite{ST08}).

Theorem \ref{T102} presents a similar phenomenon to Perelman's scalar curvature bound, albeit with different background and proof techniques in our specific setting. It is worth noting that the scalar curvature bound in Theorem \ref{T102} assumes a complex dimension of $2$, although we conjecture that this assumption may be unnecessary.

The proof of Theorem \ref{T102} is based on the idea that if the scalar curvature tends to infinity, it does so uniformly. This, in turn, leads to the divergence of the scalar curvature average, which is a contradiction. The uniform behavior of the scalar curvature necessitates a thorough understanding of the geometry in the vicinity of points with high scalar curvature. This understanding is achieved through the following canonical neighborhood theorem.

\begin{thm}[\textbf{Canonical Neighborhood Theorem}]
\label{T101}
For any positive constants $A,\ep$ and $\delta \in (0,1)$, there exist positive constants $\sigma=\sigma(\ep,A)$ and $L=L(\ep,\delta,A)$ satisfying the following property.

Let $(M^2,g,J,f)$ be a K\"ahler Ricci shrinker surface with $\boldsymbol{\mu}(g) \ge -A$. Then the following statements hold for $x \in M$.
\begin{enumerate}[label=\textnormal{(\Alph{*})}]
\item If $R(x) \ge \delta f(x)$ and $f(x) \ge L$, then $(M,g,x)$ is $\ep$-close to a pointed K\"ahler steady soliton orbifold with isolated singularities. Moreover, $R(y) \ge L^{-1} f(x)$ for any $y \in B(x,\ep^{-1} R^{-\frac 1 2} (x))$.

\item If $R(x)+R^{-1}(x) \le \sigma f(x)$, then $(M,g,J,x)$ is $\ep$-close to the K\"ahler Ricci flow $( \CP^1 \times \C ,g_c(t):=g_{\CP^1}(t) \times g_E,J_c) $.
\end{enumerate}
\end{thm}

Theorem~\ref{T101} serves as the technical core of this article. It is important to note that this result holds uniformly. The quantity $\boldsymbol{\mu}(g)$ represents Perelman's functional (cf. equation (\ref{eqn:mufunctional})).

The notion of being $\epsilon$-close to a K\"ahler steady soliton orbifold or $( \CP^1 \times \C ,g_c(t),J_c) $ can be found in Definition \ref{def:close2} and Definition \ref{def:close1}, respectively. In essence, it implies that a neighborhood around the base point is close to a K\"ahler steady soliton orbifold in the Gromov-Hausdorff sense or close to $\CP^1 \times \mathbb{C}$ in the smooth sense, upon appropriate rescaling.

These two situations are distinguished by whether the ratio $R/f$ at the base point is small. To provide an overview of the proof of Theorem \ref{T101}, although at the cost of sacrificing some rigorous definitions, we can summarize it as follows.

\begin{proof}[Outline of proof of Theorem~\ref{T101}]
Note that both parts (A) and (B) can be reformulated as the weak-compactness of K\"ahler Ricci shrinkers after appropriate rescaling. Part (A) states that for any sequence of points $x_i$ tending to infinity, with the ratio $R(x_i)/f_i(x_i)$ bounded from below, rescaling the sequence by the scalar curvature $R(x_i)$ leads to convergence of the pointed space sequence at $x_i$ to a steady soliton orbifold in the Gromov-Hausdorff topology. Furthermore, this convergence can be enhanced to the smooth topology away from singularities. Part (B) demonstrates that if $x_i$ tends to infinity and $(R(x_i) + R^{-1}(x_i))/f_i(x_i)$ tends to zero, then a suitably rescaled spacetime converges to a cylindrical K\"ahler Ricci flow in the smooth topology.

It is worth noting that there are weak-compactness theories for solitons available in the literature, such as those presented in \cite{HM11}, \cite{HM15}, \cite{LLW21}, and \cite{Bam20c}. However, none of these theories can be directly applied in the current context. Significant improvements of these weak-compactness theories are necessary to derive meaningful geometric conclusions.

The proof of part (A) is an improvement of the earlier weak-compactness theories in \cite{HM15} and \cite{LLW21}, where $f_i(x_i)$ are uniformly bounded and the scales remain fixed. In the current scenario, we are addressing the case where $f_i(x_i)$ tends to infinity and the scales involved are very small. Due to this scale consideration, the uniform no-local-collapsing theorem proved in \cite{LW20} plays a crucial role in our proof. Additionally, we modify the local conformal transformation technique from \cite{LLW21} to achieve bootstrapping and improve regularity.

The proof of part (B) is more involved. Our proof utilizes the $\IF$-compactness developed by Bamler. However, the $\IF$-compactness typically deals with a weak topology, similar to how the compactness in the mean curvature flow theory is handled by the Brakke flow \cite{Brakke}. To improve the convergence in the smooth topology, significantly more effort is required.

To achieve this, we rely on the following key ingredients:

\begin{itemize} 
\item The first is the local scalar curvature estimate, which originates from our previous work~\cite{LLW21}. This estimate ensures control over points with bounded reduced distance (in the sense of Perelman~\cite{Pe1}) to given points or related $H$-center in the terminology of~\cite{Bam20a}.
\item The second is obtaining a splitting direction of the limiting flow, thus reducing the dimension of the flow by one. The splitting direction comes from $\nabla f$ after proper rescaling. It is important to note that we analyze the estimate of $\nabla f$ in spacetime, rather than solely on a fixed time slice.

\item The third is to demonstrate that the limiting flow corresponds to a conventional ancient Ricci flow solution. By combining the symmetry induced by $J$ with the tangent flow argument, we are able to eliminate the singularities in the limiting flow. Additionally, employing a localized maximum principle argument similar to the one used in \cite{CBL07}, we establish that the limiting flow possesses nonnegative sectional curvature.

Using Toponogov's comparison geometry, as exemplified in techniques employed in \cite{Pe1}, we deduce that the limiting flow has bounded sectional curvature on each time slice. Consequently, the flow can be classified as a conventional ancient K\"ahler Ricci flow.

\item The fourth involves the classification of ancient K\"ahler Ricci flow solutions of complex dimension two, as developed by the first-named author and their collaborator \cite{CL20}. It is established that any non-flat ancient K\"ahler Ricci flow solution with an $\mathbb{R}$-factor splitting must be $(\CP^1 \times \mathbb{C}, g_c(t), J_c)$, which represents the canonical ancient solution on the cylinder $\CP^1 \times \mathbb{C}$.
\end{itemize}

The proof of part (B) becomes relatively straightforward when considering the key ingredients described above. We start with the given sequence of pointed K\"ahler Ricci shrinkers $(M_i, g_i, J_i, f_i, x_i)$ and study the associated pointed K\"ahler Ricci flows $(M_i, g_i(t), J_i, (x_i, 0))_{t \leq 0}$.

Treating the pointed K\"ahler Ricci flow as a metric flow coupled with the conjugate heat flow from $(x_i, 0)$, we appropriately rescale the metric flow pairs to demonstrate $\IF$-convergence to a limiting metric flow pair (cf. Theorem \ref{Fcom}). However, the limiting metric flow typically contains singularities, and the $\IF$-convergence is weaker than Gromov-Hausdorff convergence.

To eliminate the singular set and improve the convergence to the smooth topology, we utilize the local scalar curvature estimate, which partially ensures compatibility between $\IF$-convergence and pointed Gromov-Hausdorff convergence. Additionally, considering the sufficiently small ratio $R/f$, we construct a splitting direction induced by $\nabla f_i$ on the limiting metric flow.

Once we establish that the limiting metric flow exhibits a line splitting, we utilize the K\"ahler structure and various techniques to prove that the singular set is empty. This further confirms that the limiting metric flow corresponds to a conventional ancient K\"ahler Ricci flow solution with nonnegative bisectional curvature. In particular, the limiting metric flow is identified as $(\CP^1 \times \mathbb{C}, g_c(t))$, and the convergence is smooth.
\end{proof}

From the aforementioned discussion, it is evident that part (B) can be interpreted as a theorem focused on removing singularities and improving regularity. This approach follows a classical strategy employed in various theories. For instance, in the context of mean curvature flow theory, a similar strategy is employed in \cite{LHZW19}. By utilizing Brakke flow compactness, the central aspect of the proof in \cite{LHZW19} involves enhancing convergence to be smooth away from a curve. Subsequently, $L$-stability is applied to eliminate the singular curve and ensure regularity.

Based on Theorem~\ref{T101}, we are now prepared to provide a more detailed explanation of the proof of Theorem \ref{T102}.

\begin{proof}[Outline of proof of Theorem \ref{T102}]
Consider a complete K\"ahler Ricci shrinker surface $(M, g, f, J, p)$, where $p$ is a point at which $f$ attains its minimum value. The crucial step is to examine the behavior of the scalar curvature $R$ along the flow $\phi^s$ generated by $\nabla f/|\nabla f|^2$.

Suppose $\Sigma$ is a connected component of some level set $\Sigma(s_0)$ of $f$. If the scalar curvature $R$ remains smaller than $1$ on $\phi^s(\Sigma)$ for all $s \geq s_0$, then the union of these level sets, $\bigcup_{s \geq s_0} \phi^s(\Sigma)$, forms an end of the K\"ahler Ricci shrinker. Since the end is unique (cf. Theorem \ref{thm:oneend}), this implies that the scalar curvature $R$ is bounded.

Alternatively, if the scalar curvature $R$ exceeds $1$ by a detectably small amount at a certain point on $\Sigma$, we can demonstrate that any point on $\Sigma$ along the flow $\phi^s$ will eventually satisfy the conditions of Theorem \ref{T101}(A). Consequently, outside a compact set, the scalar curvature is comparable to $f$.

It is worth noting that as $s$ approaches infinity, the volume of the ball $B(p, \sqrt{s})$ is at least $\epsilon \sqrt{s}$; see Proposition \ref{prop:volume} and the relevant reference. Consequently, the average of the scalar curvature over $B(p, \sqrt{s})$ tends to infinity, which contradicts the integration of the equation $R+\Delta f=2$.
\end{proof}

This paper is organized as follows. Section 2 discusses the elementary properties of Ricci shrinkers and their associated Ricci flows. In Section 3, we prove part (A) of the canonical neighborhood theorem, which is modeled on a steady Ricci soliton orbifold.
Section 4 focuses on part (B) of the canonical neighborhood theorem, which is modeled on the canonical ancient K\"ahler Ricci flow solution on $\CP^1 \times \C$. Building upon the canonical neighborhood theorem, we establish the proof of Theorem \ref{T102} in Section 5.
\\
\\
{\bf Acknowledgements}: 

Yu Li is supported by YSBR-001, NSFC-12201597, and research funds from USTC (University of Science and Technology of China) and CAS (Chinese Academy of Sciences). Bing Wang is supported by YSBR-001, NSFC-11971452, NSFC-12026251, NSFC-12431003, and research funds from USTC. The authors thank Prof. Xiuxiong Chen, Prof. Jingchen Hu, Prof. Claude LeBrun, and Prof. Song Sun for their valuable insights and comments. The authors also appreciate the anonymous referees for their thoughtful feedback, which improved the clarity and presentation of the paper.

\section{Preliminaries}
A K\"ahler Ricci shrinker $(M^m, g,J, f)$ is a Ricci shrinker on a K\"ahler manifold $(M^m,g,J)$, where we denote the complex structure by $J$, the complex dimension by $m$ and the real dimension by $n=2m$. Under local holomorphic coordinates, \eqref{E100} is equivalent to
\begin{align} \label{E100c}
R_{i \bar j}+f_{i \bar j}=\frac{1}{2}g_{i \bar j}, \quad f_{ij}=f_{\bar i \bar j}=0.
\end{align}
It is well known from \eqref{E100c} that $J \na f $ is a Killing and real holomorphic vector field on $M$.

For any Ricci shrinker $(M^n,g,f)$, the scalar curvature $R\ge 0$ \cite[Corollary $2.5$]{CBL07} and $R>0$ unless $(M^n,g)$ is isometric to the Gaussian soliton $(\R^n,g_E)$, by the strong maximum principle.

With the normalization \eqref{E101}, the entropy is defined as
\begin{align} 
\boldsymbol{\mu}=\boldsymbol{\mu}(g)\coloneqq \log \int\frac{e^{-f}}{(4\pi)^{n/2}}\, dV. \label{eqn:mufunctional}
\end{align}

Note that $e^{\boldsymbol{\mu}}$ is uniformly comparable to the volume of the unit ball $B(p,1)$, see \cite[Lemma $2.5$]{LLW21}. 
It was proved in \cite{CN09} (see also \cite[Proposition $3$]{LW20}) that $\boldsymbol{\mu}$ agrees with Perelman's celebrated entropy functional $\boldsymbol{\mu}(g,1)$; see~\cite{Pe1},~\cite{BW17}. 
In fact, $\boldsymbol{\mu}$ is the optimal log-Sobolev constant for all scales; see \cite[Theorem $1$]{LW20}. For the detailed discussions of $\boldsymbol{\mu}$-functional on Ricci shrinkers and its properties, we refer the readers to \cite[Section $5$]{LW20}.

For the potential function $f$, we have the following fundamental estimates.

\begin{lem}[\cite{CZ10} \cite{HM11}]
\label{L201}
Let $(M^n,g,f)$ be a Ricci shrinker. Then there exists a point $p \in M$ where $f$ attains its infimum, and $f$ satisfies the quadratic growth estimate
\begin{align*}
\frac{1}{4}\left(d(x,p)-5n \right)^2_+ \le f(x) \le \frac{1}{4} \left(d(x,p)+\sqrt{2n} \right)^2
\end{align*}
for all $x\in M$, where $a_+ :=\max\{0,a\}$.
\end{lem}

We will always denote a minimum point of $f$ by $p$, which is regarded as the Ricci shrinker's base point. Moreover, Lemma \ref{L201} implies that $f$ is almost $d^2(p,\cdot)/4$. For simplicity, we define $\Sigma(s,t):=\{x\in M \mid s \le f(x) \le t\}$ and $\Sigma(t):=\Sigma(t,t)$.

Recall that any Ricci shrinker $(M^n,g,f)$ can be considered as a self-similar solution to the Ricci flow. Let ${\psi^t}: M \to M$ be a family of diffeomorphisms generated by $\dfrac{1}{1-t}\nabla f$ and $\psi^{0}=\text{id}$.
In other words, we have
\begin{align} 
\frac{\partial}{\partial t} {\psi^t}(x)=\frac{1}{1-t}\nabla f\left({\psi^t}(x)\right). \label{E201a}
\end{align}
It is well known that the rescaled pull-back metric $g(t)\coloneqq (1-t) (\psi^t)^*g$ satisfies the Ricci flow equation for any $-\infty <t<1$
\begin{align} 
\partial_t g=-2 \Rc_{g(t)} \quad \text{and} \quad g(0)=g. \label{E201ax}
\end{align}
Moreover,  for any $-\infty <t<1$, the pull-back function $F(t):=(1-t)(\psi^t)^*f$ satisfies the following identities
\begin{align} 
&(1-t)\Rc_{g(t)}+\text{Hess}_{g(t)}\,F(t)=\frac{1}{2}g(t), \label{E201aa} \\
&\partial_t F=-(1-t)R_{g(t)}, \label{E201ab}\\
&(1-t)^2 R_{g(t)}+|\na_{g(t)} F|^2_{g(t)}=F. \label{E201ac}
\end{align}
We refer the readers to \cite[Section 2]{LW20} for the proof of the above identities. 

\begin{defn}
	For any Ricci shrinker, the Ricci flow defined in \eqref{E201ax} is called the associated Ricci flow. Any time-shifting and rescaling of the associated Ricci flow is called the Ricci flow induced by the Ricci shrinker.
\label{dfn:RA19_1}	
\end{defn}

\begin{exmp}[Standard model] \label{ex:cyl}
The standard product $(\CP^1 \times \C,g_c:=g_{\CP^1}\times g_E,J_c, f_c)$ is a K\"ahler Ricci shrinker with scalar curvature to be identically $1$ and the potential function $f_c=|z|^2/4+1$, where $z$ is the holomorphic coordinate of the $\C$-factor. $ J_c$ denotes the unique complex structure compatible with $g_c$. The associated Ricci flow is $(\CP^1 \times \C,g_c(t))_{t \in (-\infty,1)}$, where
\begin{align*} 
g_c(t)= (1-t)g_{\CP^1} \times g_E, \quad g_c(0)=g_c.
\end{align*}
The scalar curvature evolves under the Ricci flow by
\begin{align} \label{eq:scalar}
\partial_t R= R^2.
\end{align}
\end{exmp}

Next, we have the following definition, which measures how close a parabolic neighborhood of a point in a K\"ahler Ricci flow is to another.

\begin{defn}[$\ep$-close to a pointed K\"ahler Ricci flow] \label{def:close1}
Let $\lc M_i,g_i(t),J_i,(x_i,t_i)\rc,\,i=1,2$ be two pointed K\"ahler Ricci flow solutions. We say $\lc M_1,g_1(t),J_1,(x_1,t_1) \rc$ is $\ep$-close to $\lc M_2,g_2(t),J_2,(x_2,t_2)\rc$
 if there exist open neighborhoods $U_i$ such that $B_{g_1(t_1)} \lc x_1,(\ep^{-1}-\ep)r \rc \subset U_1 \subset \bar U_1 \subset B_{g_1(t_1)}(x_1,\ep^{-1} r)$, where $R_{g_1}(x_1,t_1)=r^{-2}$, and $B_{g_2(t_2)}(x_2,\ep^{-1}-\ep) \subset U_2 \subset \bar U_2 \subset B_{g_2(t_2)}(x_2,\ep^{-1})$. Moreover, there exists a diffeomorphism $\varphi:U_2 \to U_1$ such that $\varphi(x_2)=x_1$ and
\begin{align*} 
\sup_{ U_1 \times [t_2-\ep^{-1},t_2]} \lc |\varphi^*J_1-J_2|+|\varphi^* \tilde g_1-g_2|^2+\sum_{i=1}^{[\ep^{-1}]} |\na^i_{g_2} (\varphi^*\tilde g_1)|^2\rc \le \ep^2.
\end{align*} 
Here, $\tilde g_1(t)=r^{-2}g_1 \lc r^2(t-t_2)+t_1 \rc$ is the rescaled K\"ahler Ricci flow of $g_1(t)$. For any K\"ahler Ricci shrinker $(M^n,g,J,f)$, we say $(M,g,J,x)$ is $\ep$-close to a pointed K\"ahler Ricci flow if $(M,g(t),J,(x,0))$ is, for the associated K\"ahler Ricci flow $g(t)$. 
\end{defn}

The concept of $\ep$-close to a pointed Ricci flow is similarly defined. We will omit the base point $(x_2,t_2)$, if $(M_2,g_2(t))$ is homogeneous, e.g., standard $\CP^1 \times \C$ and $S^2 \times \R$.

\section{Canonical neighborhood theorem—part A}


Following \cite{LWs1}, we have the following definition. 

\begin{defn}
Let $\mathcal M(A)$ be the family of Ricci shrinkers $(M^n,g,f)$ satisfying $\boldsymbol{\mu}(g) \ge -A$.
\label{dfn:201}
\end{defn}

Notice that the Definition \ref{dfn:201} is equivalent to the volume non-collapsing condition on $B(p,1)$, see \cite[Lemma $2.5$]{LLW21}. Under the lower bound of the entropy, one has the following useful no-local-collapsing inequality.

\begin{thm} \emph{(\cite[Theorem 22]{LW20})}\label{T201a}
For any Ricci shrinker $(M^n,g,f) \in \mathcal M(A)$, there exists a constant $\kappa=\kappa(A,n)>0$ such that for any $B(q,r) \subset M$ with $R \le r^{-2}$, we have
\begin{align*}
|B(q,r)| \ge \kappa r^n.
\end{align*}
\end{thm}

Another important property of the moduli space $\mathcal M(A)$ is weak-compactness. More precisely, we have

\begin{thm}[\cite{HM11} \cite{HM15}]
Let $(M_i^4, g_i,f_i) \in \mathcal M(A)$ be a sequence of Ricci shrinkers. Then by taking subsequence if necessary, we have
\begin{align*}
(M^4_i, g_i, f_i,p_i) \longright{pointed-\hat{C}^{\infty}-Cheeger-Gromov} (X_{\infty}, g_{\infty}, f_{\infty},p_{\infty}), 
\end{align*}
where $p_i$ is a minimum point of $f_i$ and $(X_{\infty}, g_{\infty},f_{\infty}, p_{\infty})$ is a smooth Ricci shrinker orbifold with isolated singularities.
\end{thm}

Intuitively speaking, ``pointed-$\hat{C}^{\infty}$-Cheeger-Gromov" means that the convergence is smooth away from the orbifold singularities. For the precise meaning, see the discussion after \cite[Theorem $2.6$]{LWs1}. Notice that the limit space carries a Ricci shrinker structure 
\begin{align*} 
\Rc_{g_{\infty}}+\text{Hess}_{g_{\infty}} \,f_{\infty}=\frac{1}{2}g_{\infty},
\end{align*}
by lifting to the orbifold local chart if necessary. We remark that the higher dimensional generalization of the weak-compactness theory was established in~\cite{LLW21} and~\cite{HLW}. 

Now, we prove the following convergence result in $\mathcal M(A)$ if the base points are moving to infinity.

\begin{thm}
\label{T203}
Let $(M_i^4, g_i,f_i) \in \mathcal M(A)$ be a sequence of Ricci shrinkers. Then by taking subsequence if necessary, we have
\begin{align}
(M_i, \tilde g_i, \tilde f_i, x_i) \longright{pointed-\hat{C}^{\infty}-Cheeger-Gromov} (Y_{\infty}, g_{\infty}, f_{\infty}, x_{\infty}), \label{E202b} 
\end{align}
where $s_i:=f_i(x_i) \to \infty$, $\tilde g_i=s_i g_i$, $\tilde f_i=f_i-s_i$ and $(Y_{\infty}, g_{\infty},f_{\infty}, x_{\infty})$ is a smooth Ricci steady soliton orbifold with isolated singularities. In addition, $R_{\infty}>0$ unless $(Y_{\infty},g_{\infty})=(\R^4,g_E)$.

Morever, if $(M_i, g_i,f_i)$ are K\"ahler Ricci shrinkers, then $(Y_{\infty}, g_{\infty}, f_{\infty})$ is a K\"ahler Ricci steady soliton orbifold.
\end{thm}

First, we prove the following proposition.

\begin{prop} 
For any $L>1$, there exists a constant $C=C(A,L)>1$ such that the following statements hold if $i$ is sufficiently large.
\begin{enumerate}[label=\textnormal{(\roman{*})}]
\item For any $y \in B_{\tilde g_i}(x_i,2L)$,
\begin{align}
|\tilde f_i(y)| \le 3L \quad \text{and} \quad R_{\tilde g_i}(y) \le 2.
\label{E203xa}
\end{align}

\item \emph{(Volume estimates)} For any $B_{\tilde g_i}(y,r) \subset B_{\tilde g_i}(x_i,2L)$ with $r \le 1$, we have
\begin{align} \label{E203xb}
C^{-1}r^4 \le |B_{\tilde g_i}(y,r)| \le Cr^4.
\end{align}

\item \emph{(Sobolev inequality)} For any Lipschitz function $u$ compactly supported in $B_{\tilde g_i}(x_i,L)$,
\begin{align} \label{E203xd}
\lc \int u^4 \,dV_{\tilde g_i} \rc^{\frac 1 2} \le C \int |\na_{\tilde g_i} u|^2 \, dV_{\tilde g_i}.
\end{align}

\item \emph{(Energy control)} The $L^2$-integral of the curvature satisfies
\begin{align} \label{E203xc}
\int_{B_{\tilde g_i}(x_i,L)} |\Rm_{\tilde g_i}|^2 \,dV_{\tilde g_i} \le C.
\end{align}
\end{enumerate}
For simplicity, all norms and volume are taken concerning $\tilde g_i$.
\end{prop}

\begin{proof}
In the proof, all constants $C_i$ depend on $A$ and $L$.

From \eqref{E101}, we have 
\begin{align}
R_{\tilde g_i}+|\na_{\tilde g_i} f_i|^2_{\tilde g_i}=s_i^{-1}(R_{g_i}+|\na_{g_i} f_i|^2_{g_i})=s_i^{-1}f_i. \label{E204xa}
\end{align}
From \eqref{E204xa} and the fact that $R_{\tilde g_i} \ge 0$, it is clear that $\sqrt{f_i}$ is $\frac{1}{2}s_i^{-\frac1 2}$-Lipschitz with respect to $\tilde g_i$. Therefore, for any $y \in B_{\tilde g_i}(x_i,2L)$,
\begin{align}\label{E204xb}
\sqrt{s_i}-\frac{L}{\sqrt{s_i}}\le \sqrt{f_i(y)} \le \sqrt{s_i}+\frac{L}{\sqrt{s_i}}.
\end{align}
Since $s_i \to \infty$, \eqref{E203xa} follows immediately from \eqref{E204xb} and \eqref{E204xa}.

Next, we consider the conformal metric $\bar g_i:=e^{-\tilde f_i} \tilde g_i$. It follows from \cite[Lemma 3.5]{LLW21} that 
\begin{align}
|\Rc_{\bar g_i}|_{\bar g_i} \le C_1 \quad \text{on} \quad B_{\tilde g_i}(x_i,2L). \label{E204xc}
\end{align}
Since $\bar g_i$ is uniformly comparable to $\tilde g_i$ by \eqref{E203xa}, the upper bound of \eqref{E203xb} follows immediately from the corresponding estimate for $\bar g_i$, by the Bishop-Gromov volume comparison theorem. Moreover, the lower bound of \eqref{E203xb} can be derived from \eqref{E203xa} and Theorem \ref{T201a}.

From \eqref{E204xc} and the isoperimetric constant estimate~\cite{Cro80}, it is well-known that $L^2$-Sobolev inequality for $\bar g_i$ holds uniformly on $B_{\tilde g_i}(x_i,L)$. Therefore, \eqref{E203xd} is immediate since $\bar g_i$ and $\tilde g_i$ are uniformly comparable.

To prove (iv), we first derive the $L^2$ estimate for $\Rc$. From \eqref{E100},
\begin{align}
\text{Hess}_{\tilde g_i} \tilde f_i+\Rc_{\tilde g_i}=\text{Hess}_{g_i} f_i+\Rc_{g_i}=\frac{g_i}{2}=\frac{\tilde g_i}{2s_i}.
\label{E204xd}
\end{align}

We fix a cutoff function $\psi$ on $\R$ such that $\psi(t)=1$ for $t \le 1$ and $\psi(t)=0$ for $t \ge 2$. By defining $\eta(x)=\psi(\frac{d_{\tilde g_i}(x_i,x)}{L})$, we compute
\begin{align*} 
\int \eta^2|\Rc_{\tilde g_i}|^2e^{-\tilde f_i} \,dV_{\tilde g_i}=&\int \eta^2\la \frac{\tilde g_i}{2s_i}-\text{Hess}\,\tilde f_i,\Rc_{\tilde g_i} \ra e^{-\tilde f_i} \,dV_{\tilde g_i} \\
=& \int \lc \frac{1}{2s_i}\eta^2 R_{\tilde g_i}+2\eta \Rc_{\tilde g_i}(\na_{\tilde g_i} \eta,\na_{\tilde g_i} \tilde f_i) \rc e^{-\tilde f_i} \,dV_{\tilde g_i} \\
\le & \int \lc \frac{1}{2s_i}\eta^2 R_{\tilde g_i}+\frac{1}{2}\eta^2|\Rc_{\tilde g_i}|^2+2|\na_{\tilde g_i} \eta|^2|\na_{\tilde g_i} \tilde f_i|^2 \rc e^{-\tilde f_i} \,dV_{\tilde g_i},
\end{align*}
where the norm and inner product are taken with respect to $\tilde g_i$ and we have used $\text{div}_{\tilde g_i}(\Rc_{\tilde g_i}\,e^{-\tilde f_i})=0$ for the second line from \eqref{E204xd}.
Consequently, by \eqref{E203xa} and \eqref{E203xb}, we have
\begin{align*} 
\int \eta^2|\Rc_{\tilde g_i}|^2e^{-\tilde f_i} \,dV_{\tilde g_i} \le \int \lc s_i^{-1}\eta^2 R_{\tilde g_i}+4|\na_{\tilde g_i} \eta|^2|\na_{\tilde g_i} \tilde f_i|^2\rc e^{-\tilde f_i} \,dV_{\tilde g_i} 
\le C_2.
\end{align*}
Therefore, one obtains
\begin{align}
\int_{B_{\tilde g_i}(x_i,L)}|\Rc_{\tilde g_i}|^2 \,dV_{\tilde g_i} \le C_3.
\label{E204xe}
\end{align}

For the conformal metric $\bar g_i$, we have
\begin{align}
\Rm_{\bar g_i}=e^{-\tilde f_i}\left[\Rm_{\tilde g_i}+\frac{1}{2}\left(\frac{d\tilde f_i \otimes d\tilde f_i}{2}+\frac{\tilde g_i}{2}\lc1-\frac{|\nabla_{\tilde g_i} \tilde f_i|^2}{2}\rc-\frac{\Rc}{2}\right)\KN \tilde g_i \right],\label{E204xf}
\end{align}
where the proof and the definition of the Kulkarni-Nomizu product $\KN$ can be found in \cite[Theorem 1.165]{Besse}. By using a covering argument, it follows from \eqref{E204xc}, \eqref{E203xb} and \cite[Theorem 1.13]{CN15} that
\begin{align}\label{E204xg}
\int_{B_{\tilde g_i}(x_i,L)} |\Rm_{\bar g_i}|^2 \,dV_{\tilde g_i} \le C_4.
\end{align}

From \eqref{E204xf}, we have on $B_{\tilde g_i}(x_i,L)$,
\begin{align}\label{E204xh}
|\Rm_{\tilde g_i}|^2 \le C_5\lc |\Rm_{\bar g_i}|^2+|\na_{\tilde g_i} \tilde f_i|^4+|\Rc_{\tilde g_i}|^2+1 \rc \le C_6 \lc|\Rm_{\bar g_i}|^2+|\Rc_{\tilde g_i}|^2+1 \rc,
\end{align}
since $|\na_{\tilde g_i} \tilde f_i|$ are uniformly bounded by \eqref{E204xa}.
Combining \eqref{E204xe}, \eqref{E204xg} and \eqref{E204xh}, we obtain
\begin{align*}
\quad \int_{B_{\tilde g_i}(x_i,L)} |\Rm_{\tilde g_i}|^2\,dV_{\tilde g_i} \le C_7.
\end{align*}

In sum, the proof is complete.
\end{proof}

\emph{Proof of Theorem \ref{T203}}:
Recall from \eqref{E204xd} and \eqref{E204xa}, we have 
\begin{align}
&\text{Hess}_{\tilde g_i} \tilde f_i+\Rc_{\tilde g_i}=\frac{\tilde g_i}{2s_i},
\label{E203a} \\
&R_{\tilde g_i}+|\na_{\tilde g_i} \tilde f_i|^2_{\tilde g_i}=1+\frac{\tilde f_i}{s_i}. \label{E203b}
\end{align}

It follows from \eqref{E203xa} and \eqref{E203xb} that
\begin{align}
|B_{\tilde g_i}(x_i,1)|_{\mu_i} \ge v_0
\label{E203d}
\end{align}
for some $v_0=v_0(A)>0$, where $\mu_i:=e^{-\tilde f_i}dV_{\tilde g_i}$ is the weighted measure. Combining \eqref{E203a}-\eqref{E203d}, it follows from the main result \cite[Theorem $1.1$]{LLW21} and its variant \cite[Theorem $10.2$]{LLW21} that by taking a subsequence if necessary,
\begin{align}
(M_i, \tilde g_i, \tilde f_i, x_i) \longright{pointed-\hat{C}^{\infty}-Cheeger-Gromov} \left(Y_{\infty}, g_{\infty}, f_{\infty}, x_{\infty} \right). \label{E204}
\end{align}
where $(Y_{\infty}, d_{\infty})$ is a length space, $f_{\infty}$ is a Lipschitz function on $(Y_{\infty}, d_{\infty})$. The space $Y_{\infty}$ has a natural regular-singular decomposition $Y_{\infty}=\mathcal{R} \cup \mathcal{S}$ satisfying
\begin{itemize}
\item[(a).] The singular part $\mathcal{S}$ is a closed set of Minkowski dimension $0$. 

\item[(b).] The regular part $\mathcal{R}$ is an open manifold with smooth metric $g_{\infty}$ satisfying Ricci steady soliton equation
\begin{align}
\Rc_{g_{\infty}} + \text{Hess}_{g_{\infty}} f_{\infty}=0 \quad \text{and} \quad R_{g_{\infty}}+|\na_{g_{\infty}} f_{\infty}|^2=1. \label{E204a}
\end{align}
\end{itemize}

Note that \eqref{E204} implies that away from $\mathcal{S}$, the convergence is smooth. This follows from the bootstrapping argument based on \eqref{E203a} and \eqref{E203b}, see \cite[Section 7]{LLW21}. Moreover, \eqref{E204a} follows from \eqref{E203a} and \eqref{E203b}, by taking the limit.

With the Sobolev inequality \eqref{E203xd} and the energy control \eqref{E203xc}, one can prove similarly as in \cite{HM11}, following the well-known methods of \cite{BKN89}\cite{An89}, that $Y_{\infty}$ is a smooth orbifold with isolated singularities. We sketch the proof as follows.

From the equation \eqref{E203b}, we have the following elliptic equation (see \cite[Lemma 2.1]{PW10})
\begin{align}
\Delta_{\tilde f_i} \Rm_{\tilde g_i}=s_i^{-1} \Rm_{\tilde g_i}+\Rm_{\tilde g_i}*\Rm_{\tilde g_i},\label{E205a}
\end{align}
where $\Delta_{\tilde f_i}=\Delta_{\tilde g_i}-\la \na_{\tilde g_i} \tilde f_i, \na_{\tilde g_i} \cdot \ra$. Since $s_i \to \infty$, one can derive from \eqref{E203xd}, \eqref{E203xc} and \eqref{E205a} the $\ep$-regularity as \cite[Lemma 3.3]{HM11}. That is, there exists $\ep=\ep(A,L)>0$ and $C_l=C_l(A,L)>0$ such that
\begin{align*}
\sup_{B_{\tilde g_i}(y,r)}|\na_{\tilde g_i}^l \Rm_{\tilde g_i}| \le \frac{C_l}{r^{l+2}} \|\Rm_{\tilde g_i}\|_{L^2(B_{\tilde g_i}(y,2r))}
\end{align*}
for any $l \ge 0$, provided that $B_{\tilde g_i}(y,2r) \subset B_{\tilde g_i}(x_i,L)$ and $\|\Rm_{\tilde g_i}\|_{L^2(B_{\tilde g_i}(y,2r))} \le \ep$.

Based on the $\ep$-regularity, one immediately concludes that $\mathcal S$ is isolated. Moreover, any $x \in \mathcal S$ has a unique tangent cone $(\R^4/\Gamma,g_E)$, where $\Gamma \subset O(4)$ is a finite subgroup acting freely on $S^3$. Now, one can follow the proof of \cite[Theorem 1.1]{HM11} to first show that there exists a $C^0$-orbifold chart around $x$. Taking the limit of \eqref{E205a}, one obtains
\begin{align}
\Delta_{f_{\infty}} \Rm_{ g_{\infty}}=\Rm_{g_{\infty}}*\Rm_{g_{\infty}} \label{E205aa}
\end{align}
on the regular part $\RR$. Based on the local $L^2$ bound of $\Rm_{g_{\infty}}$ around $x$, one can apply the Moser iteration to \eqref{E205aa} and follow the strategy of \cite[Theorem 5.1]{BKN89} to show $\Rm_{g_{\infty}}$ is uniformly bounded near $x$. Therefore, one can construct a $C^{\infty}$ orbifold chart around $x$ as done in \cite[Theorem 5.1]{BKN89}.

In particular, \eqref{E204a} hold globally on $Y_{\infty}$, by lifting to the orbifold chart if necessary. Therefore, $(Y_{\infty},g_{\infty},f_{\infty})$ is a smooth steady soliton orbifold with isolated singularities.

Notice that $R_{g_{\infty}} \ge 0$ on $Y_{\infty}$ and satisfies (\cite[Lemma 2.1]{PW10})
\begin{align*}
\Delta_{f_{\infty}} R_{g_{\infty}}=-2|\Rc_{g_{\infty}}|^2.
\end{align*}
If $R_{g_{\infty}}$ vanishes at one point, we conclude that $R_{g_{\infty}} \equiv 0$ by the strong maximum principle and hence $\Rc_{g_{\infty}} \equiv 0$. From \eqref{E204a}, we have
\begin{align} \label{E205b}
\text{Hess}_{g_{\infty}} f_{\infty}=0 \quad \text{and} \quad |\na_{g_{\infty}} f_{\infty}|^2=1.
\end{align}
Clearly, \eqref{E205b} implies $(Y_{\infty},g_{\infty})=(Y' \times \R,g' \times g_E)$ and contains no singularity. Moreover, $(Y',g')$ is flat and hence isometric to $(\R^3,g_E)$ by passing Theorem \ref{T201a} to the limit $Y_{\infty}$, since the convergence \eqref{E204} is smooth in this case.

For the last statement, we assume all $(M_i,g_i,J_i,f_i)$ are K\"ahler Ricci shrinkers with complex structure $J_i$. Since the convergence \eqref{E202b} is smooth away from the orbifold points, $J_i$ converges smoothly to a complex structure $J_{\infty}$ on the regular part $\mathcal R$ of $Y_{\infty}$. 
Clearly, $g_{\infty}$ is a K\"ahler metric compatible with $J_{\infty}$, since $\na_{g_{\infty}} J_{\infty}=0$. By lifting to the orbifold chart around the singularity, $J_{\infty}$ is defined globally and hence $(Y_{\infty},g_{\infty},J_{\infty},f_{\infty})$ is a K\"ahler steady soliton orbifold.

In sum, the proof of the theorem is complete.
\qed

\begin{rem} \label{rem:conv}
For general dimension $n$, if $(M_i^n, g_i,f_i) \in \mathcal M(A)$ be a sequence of Ricci shrinkers, then by taking subsequence if necessary, we have
\begin{align*}
(M_i, \tilde g_i, \tilde f_i, x_i) \longright{pointed-\hat{C}^{\infty}-Cheeger-Gromov} (Y_{\infty}, g_{\infty}, f_{\infty}, x_{\infty}), 
\end{align*}
where $s_i:=f_i(x_i) \to \infty$, $\tilde g_i=s_i g_i$, $\tilde f_i=f_i-s_i$. The limit $(Y_{\infty}, g_{\infty},f_{\infty})$ is a 
steady soliton conifold, in the sense of \emph{\cite[Definition A.1]{LW22}}. For more information, see \emph{\cite[Appendix A]{LW22}}.
\end{rem}

Next, we have the following definition, similar to Definition \ref{def:close1}, which measures how close the neighborhood of a point in a K\"ahler Ricci shrinker is to a K\"ahler Ricci steady soliton conifold.

\begin{defn}[$\ep$-close to a pointed steady soliton orbifold]\label{def:close2}
Let $(M^4,g,f)$ be a Ricci shrinker and let $\bar x \in M$ with $R(\bar x)=r^{-2}$ . We say that $(M,g,x)$ is $\ep$-close to a pointed steady soliton orbifold $(Y,g',y)$ if, after rescaling the metric $g$ by the factor $r^{-2}$, the neighborhood $B_{g}(\bar{x},\ep^{-1}r) $ is $\ep$-close in the pointed-Gromov-Hausdorff sense to $B_{g'}(y,\ep^{-1}) $. 
\end{defn}

Now, one can easily prove (A) of Theorem \ref{T101}.

\emph{Proof of Theorem \ref{T101} (A)}: For fixed $A,\ep$ and $\delta$, if no such $L$ exists, one can obtain a sequence of K\"ahler Ricci shrinker surfaces $(M^2_i, g_i,J_i,f_i)$ and $x_i \in M_i$ such that $s_i:=f_i(x_i) \to \infty$, $R_{g_i}(x_i) \ge \delta s_i$, but $(M_i,g_i,x_i)$ is not $\ep$-close to a K\"ahler steady soliton orbifold.

By taking a subsequence if necessary, it follows from Theorem \ref{T203} that 
\begin{align*}
(M_i, \tilde g_i, J_i,\tilde f_i, x_i) \longright{pointed-\hat{C}^{\infty}-Cheeger-Gromov} (Y_{\infty}, g_{\infty}, J_{\infty},f_{\infty}, x_{\infty}), 
\end{align*}
where $(Y_{\infty}, g_{\infty}, J_{\infty},f_{\infty})$ is a K\"ahler steady soliton orbifold. If we set $r_i^{-2}=R_{g_i}(x_i)$, then $s_i \ge r_i^{-2} \ge \delta s_i$ by our assumption. Therefore, by further taking a subsequence, we may assume $s^{-1}_ir_i^{-2} \to \tau \in [\delta, 1]$ and hence 
\begin{align*}
(M_i, r_i^{-2}g_i, x_i) \longright{pointed-\hat{C}^{\infty}-Cheeger-Gromov} (Y_{\infty}, \tau g_{\infty}, x_{\infty}).
\end{align*}
Notice that $(Y_{\infty},g_{\infty})$ cannot be flat since otherwise, the convergences above are smooth, but the scalar curvature of $r_i^{-2}g_i$ at $x_i$ is identically equal to $1$.

In sum, we obtain a contradiction, and the original statement for $\ep$-closeness holds. The last estimate for the scalar curvature follows from the next result without using any K\"ahler condition.
\qed

The proof of the following proposition is essentially contained in \cite[Proposition 4.6]{LW22}. We include the adapted proof for readers' convenience.

\begin{prop}
\label{prop:lower}
For any positive constants $A$, $\delta$ and $\ep$, there exists a positive constant $L=L(\ep,\delta,A)$ satisfying the following property.

Let $(M^4,g,f)$ be a Ricci shrinker with $\boldsymbol{\mu}(g) \ge -A$. If $R(x) \ge \delta f(x)$ and $f(x) \ge L$, then $R(y) \ge L^{-1} f(x)$ for any $y \in B(x,\ep^{-1} R^{-\frac 1 2} (x))$.
\end{prop}

\begin{proof}
Suppose the conclusion fails for fixed $A,\delta$ and $\ep$. Then there exists a sequence of Ricci shrinkers $(M^4_i, g_i,f_i ,x_i) \in \MM(A)$ with $s_i:=f_i(x_i) \to \infty$ and $R_{g_i}(x_i) \ge \delta s_i$ such that 
\begin{align}
R_{g_i}(y_i) < i^{-2} s_i \label{E207a}
\end{align}
for some $y_i \in B_{g_i}(x_i,\ep^{-1} r_i)$, where $r_i^{-2}=R_{g_i}(x_i)$. From \eqref{E201ac} and our assumption,
\begin{align}
\delta s_i \le r_i^{-2} \le s_i. \label{E207aa}
\end{align}

Let $g_i(t)$ be the Ricci flow associated with the Ricci shrinker $(M_i, g_i, f_i)$. For $t \le 0$, we define 
\begin{align*}
\tilde g_i(t) \coloneqq s_i g_i(s_i^{-1}t), \quad \tilde F_i(t) \coloneqq F_i(s_i^{-1}t)-s_i, \quad \tilde \square_i \coloneqq \partial_t-\Delta_{\tilde g_i(t)}, 
\end{align*}
where $F_i(t)=(1-t)f_i(t)$. Direct calculation from \eqref{E201aa}-\eqref{E201ac} shows that 
\begin{align}
|\na_{\tilde g_i(t)} \tilde F_i(t)|^2_{\tilde g_i(t)}+(1-s_i^{-1}t)^2R_{\tilde g_i(t)}=1+ \frac{\tilde F_i(t)}{s_i} \label{E207b}
\end{align}
and
\begin{align}
\tilde \square_i \tilde F_i(t)=-\frac{2}{s_i}. \label{E207c}
\end{align}

From \eqref{E207aa} and \eqref{E207b}, there exists a constant $L_1=L_1(\ep,\delta)>1$ such that on $B_{g_i}(x_i,\ep^{-1} r_i)$,
\begin{align}
|\tilde F_i(0)| \le L_1. \label{E207cx}
\end{align}

Now we define the following parabolic balls
\begin{align*}
P^1_i& \coloneqq \{(y,t) \,\mid |\tilde F_i(y,t)| \le 2L_1, \, -2L_1 \le t \le 0\}, \\
P^2_i& \coloneqq \{(y,t) \,\mid |\tilde F_i(y,t)| \le 4L_1, \, -4L_1 \le t \le 0\} . 
\end{align*}

In the following, all positive constants $C_i$ depend only on $\delta,A$ and $\ep$, and the corresponding inequalities hold for sufficiently large $i$. It follows from \eqref{E207b} that on $ P^2_i$,
\begin{align}
0 \le R_{\tilde g_i(t)} \le C_1. \label{E207d}
\end{align}

\textbf{Claim 1:} Suppose $u$ is a nonnegative function such that 
\begin{align*}
\tilde \square_i u \le 0, \quad \textrm{on} \; P^2_i. 
\end{align*}
Then we have
\begin{align}\label{E207e}
\max_{P^1_i} u^2 \le C_2 \iint_{P^2_i} u^2 \,dV_{\tilde g_i(t)} dt
\end{align}
for some $C_2>0$.

\emph{Proof of Claim 1:} The proof follows verbatim as \cite[Lemma $9.7$]{LLW21} by using the Moser iteration. Notice that the Sobolev inequality and the control of $R_{\tilde g_i(t)}$ on $P^2_i$ are guaranteed by \cite[Theorem 1]{LW20} and \eqref{E207d} respectively. 
Moreover, the cutoff functions $\eta_k$ in the proof of \cite[Lemma $9.7$]{LLW21} can be defined similarly by using $\tilde F_i$ and $\eta_k$ can be estimated similarly by using \eqref{E207c}. We omit the details.

Next we define for any $\tau>0$,
\begin{align*}
P_i(\tau)=P^2_i\cap \{(z,t) \,\mid R_{\tilde g_i}(z,t) <\tau \}.
\end{align*}

\textbf{Claim 2:} There exists a constant $C_3>0$ such that 
\begin{align} \label{E207f}
|P_i( 2i^{-2})| \ge C_3. 
\end{align}
Here, the volume is with respect to $dV_{\tilde g_i(t)} dt$.

\emph{Proof of Claim 2:} We choose $u=(2i^{-2}-R_{\tilde g_i})_+$, where $(\cdot)_+=\max\{\cdot,0\}$, and apply \eqref{E207e} to conclude that
\begin{align} \label{E207fa}
\max_{P^1_i} (2i^{-2}-R_{\tilde g_i})^2_+ \le C_2 \iint_{P^2_i} (2i^{-2}-R_{\tilde g_i})^2_+ \,dV_{\tilde g_i(t)} dt,
\end{align}
since $\tilde \square_i (2i^{-2}-R_{\tilde g_i})_+ \le \tilde \square_i (2i^{-2}-R_{\tilde g_i}) \le 0$.

It follows from \eqref{E207a} and \eqref{E207cx} that $(y_i,0) \in P^1_i \cap P_i( i^{-2})$. Therefore, it follows from \eqref{E207fa} that

\begin{align} \label{E207fb}
i^{-4} \le C_2 \iint_{P^2_i} (2i^{-2}-R_{\tilde g_i})^2_+ \,dV_{\tilde g_i(t)} dt \le 4C_2i^{-4} |P_i(2i^{-2})|.
\end{align}

Therefore, \eqref{E207f} follows immediately from \eqref{E207fb}.

It is clear from \eqref{E207f} that for any $i$ sufficiently large, there exists $t_i \in [-2L_1,0]$ such that 
\begin{align*}
|\Sigma_i|_{dV_{\tilde g_i(t_i)}} \ge C_4,
\end{align*}
where $\Sigma_i=M_i \times \{t_i\} \cap P_i( 2i^{-2})$ and $C_4=C_3/2L_1$. Now, we define 
\begin{align*}
\Omega_i:=\psi_i^{s_i^{-1}t_i}\lc \text{pr}(\Sigma_i) \rc,
\end{align*}
where $\text{pr}$ is the projection onto $M_i$ and $\psi_i$ is the diffeomorphism defined in \eqref{E201a}. Then we compute
\begin{align} \label{E208a}
|\Omega_i|_{dV_{\tilde g_i(0)}}\ge \frac{1}{2}(1-s_i^{-1}t_i)^{2} |\Omega_i|_{dV_{\tilde g_i(0)}}= \frac{1}{2} |\Sigma_i|_{dV_{\tilde g_i(t_i)}} \ge C_5,
\end{align}
where $C_5=C_4/2$. On the other hand, for any $(z,t_i) \in \Sigma_i$, we have
\begin{align*}
\tilde F_i(z,t_i)=F_i(z,s_i^{-1}t_i)-s_i=(1-s_i^{-1}t_i) f_i(z')-s_i
\end{align*}
where $z'=\psi_i^{s_i^{-1}t_i}(z)$. Therefore, it follows from the definition of $P^2_i$ that
\begin{align} \label{E208b}
|f_i(z')-s_i| \le 10 L_1.
\end{align}

From \eqref{E202b}, by taking a subsequence if necessary, we have
\begin{align*}
(M^4_i, \tilde g_i, \tilde f_i, x_i) \longright{pointed-\hat{C}^{\infty}-Cheeger-Gromov} \left(Y_{\infty}, g_{\infty}, f_{\infty} , x_{\infty}\right), 
\end{align*}
where $\tilde g_i=s_ig_i$, $\tilde f_i:=f_i-s_i$ and $(Y_{\infty}, x_{\infty}, g_{\infty}, f_{\infty})$ is a steady soliton orbifold with isolated singularities. From \eqref{E208a} and \eqref{E208b}, there exists a sequence $z_i \in \Omega_i$ such that $z_i \to z \in \mathcal R$, otherwise it would violate the energy control \eqref{E203xc}. By the definition of $\Omega_i$, we have
\begin{align*}
R_{g_{\infty}}(z)=0.
\end{align*}

In this case, $(Y_{\infty},g_{\infty})$ must be isometric to $(\R^4,g_E)$ and the convergence is smooth. However, this contradicts $R_{g_i}(x_i) \ge \delta s_i$.

In sum, the proof is complete.
\end{proof}

\section{Canonical neighborhood theorem—part B}

Let $(M^n,g(t))_{t \in I}$ be the Ricci flow induced by a Ricci shrinker (cf. Definition~\ref{dfn:RA19_1}) in $\MM(A)$ for $I=(-\infty,0]$. It is proved in \cite[Theorem 7]{LW20} that there exists a positive heat kernel function $H(x,t,y,s)$ for $x,y \in M$ and $s,t \in I$ with $s<t$. More precisely,
\begin{align*}
\square H(\cdot,\cdot,y,s)=0, \quad \lim_{t \to s^+}H(\cdot,t,y,s)=\delta_y 
\end{align*}
and
\begin{align*}
\square^* H(x,t,\cdot,\cdot)=0, \quad \lim_{s \to t^-}H(x,t,\cdot,s)=\delta_x, 
\end{align*}
where $\square:=\partial_t-\Delta_t$ and $\square^*:=-\partial_s-\Delta_s+R$. For any $(x,t) \in M \times I$, we define the conjugate heat kernel measure $v_{x,t;s}$ by $dv_{x,t;s}(y)=K(x,t,y,s)dV_s(y)$. It follows immediately from \cite[Theorem 7 (58)]{LW20} that $v_{x,t;s}$ is a probability measure on $M$. 
In particular, $v_{x,t;t}=\delta_x$. Next, we recall the following definition of $H$-center, where the conjugate heat kernel measure is concentrated.

\begin{defn} \label{def:Hcenter}
Given a constant $H>0$, a point $(z,t) \in M \times I$ is called an $H$-center of $(x_0,t_0)\in M \times I$ if $t \le t_0$ and
\begin{align*}
\emph{\text{Var}}_t(\delta_z,v_{x_0,t_0;t})\le H(t_0-t),
\end{align*}
where the variance between two probability measures $\mu_1,\mu_2$ on $(M,g(t))$ is defined by
\begin{align*}
\emph{\text{Var}}_t(\mu_1,\mu_2):=\iint d_t^2(x,y) \,d\mu_1(x)d\mu_2(y).
\end{align*}
\end{defn}

It is proved in \cite[Proposition 3.12]{LW23} that for any $(x_0,t_0)\in M \times I$ and $t \le t_0$, there exists $z \in M$ such that $(z,t)$ is an $H_n$-center of $(x_0,t_0)$, where $H_n:=(n-1)\pi^2/2+4$. Moreover, any two $H$-centers $(z_1,t)$ and $(z_2,t)$ of $(x_0,t_0)$ satisfies $d_t(z_1,z_2) \le 2 \sqrt{H(t_0-t)}$. 
The following result ensures that the conjugate heat kernel measure is concentrated around an $H$-center (cf. \cite[Proposition 3.13]{LW23}).

\begin{prop} \label{prop:conc}
If $(z,t)$ is an $H$-center of $(x_0,t_0)$, then for any $L>0$,
\begin{align*}
v_{x_0,t_0;t} \lc B_t(z,\sqrt{LH(t_0-t)}) \rc \ge 1-\frac{1}{L}.
\end{align*}
\end{prop}

Now, we have the following heat kernel upper bound estimate from~\cite[Theorem 1.3]{LW23}.

\begin{thm}\label{thm:heatupper}
For any $\ep>0$, there exists a constant $C = C(A,\ep) >0$ such that
\begin{align}\label{E401}
H(x,t,y,s) \le \frac{C}{(t-s)^{\frac n 2}} \exp \lc -\frac{d_s^2(z,y)}{(4+\ep)(t-s)} \rc,
\end{align}
where $(z,s)$ is an $H_n$-center of $(x,t)\in M \times I$.
\end{thm}

From Theorem~\ref{thm:heatupper}, we have

\begin{prop}\label{prop:scalar}
For any $L>0$, there exists a constant $H = H(A,L,n) >0$ such that $(y,s)$ is an $H$-center of $(x,t)$, provided that $l_{(x,t)}(y,s) \le L$, where $l_{(x,t)}(y,s)$ is the reduced distance between $(x, t)$ and $(y, s)$ (see \cite[Section 7]{Pe1}). In particular, if $R(x,z) \le L/(t-z)$ for any $z \in [s,t)$, then $(x,s)$ is an $H$-center of $(x,t)$.
\end{prop}

\begin{proof}
It follows from \cite[Theorem 16]{LW20} that
\begin{align}\label{E402}
H(x,t,y,s) \ge \frac{1}{(4\pi(t-s))^{\frac n 2}} \exp \lc -l_{(x,t)}(y,s) \rc,
\end{align}

Combining \eqref{E401} for $\ep=1$ and \eqref{E402}, it is clear that
\begin{align*}
d_s^2(z,y) \le C_1(t-s)
\end{align*}
for some constant $C_1=C_1(A,L,n)>0$, where $(z,s)$ is an $H_n$-center of $(x,t)$. Consequently, we have
\begin{align*}
\text{Var}_s(\delta_y,v_{x,t;s})=\int d^2_s(y,x) dv_{x,t;s}(x) \le 2\int \lc d^2_s(z,x)+ d^2_s(z,y) \rc dv_{x,t;s}(x) \le 2(H_n+C_1)(t-s)
\end{align*}
since $v_{x,t;s}$ is a probability measure. Therefore, $(y,s)$ is an $H$-center of $(x,t)$, for $H=2(H_n+C_1)$.

The last statement follows from the definition of $l_{(x,t)}(x,s)$ since
\begin{align*}
l_{(x,t)}(x,s) &\le \frac{1}{2\sqrt{t-s}}\int_s^t \sqrt{t-z}\,R(x,z) \,dz \le \frac{L}{2\sqrt{t-s}}\int_s^t \frac{1}{\sqrt{t-z}} \,dz=L.
\end{align*}

\end{proof}

For later applications, we recall the following two-sided pseudolocality theorem for Ricci flows induced by Ricci shrinkers; see \cite[Theorem 1.6]{LW23}.

\begin{thm}[Two-sided pseudolocality theorem] \label{thm:pseudo}
For any $n \in \IN$ and $\alpha > 0$ there is an $\ep (n, \alpha) > 0$ such that the following holds.

Let $(M^n,g(t))_{t \in I}$ be a Ricci flow induced by a Ricci shrinker. Given $(x_0, t_0) \in M \times I$ and $r > 0$, if
\begin{equation*}
|B_{t_0}(x_0,r)| \geq \alpha r^n, \qquad |\Rm| \leq (\alpha r)^{-2} \quad \text{on} \quad B_{t_0}(x_0,r),
\end{equation*}
then
\[ |\Rm|(x,t) \leq (\ep r)^{-2} \quad \text{for} \quad d_{t_0}(x,x_0)<(1-\alpha) r \quad \text{and} \quad t\in [-(\ep r)^2),(\ep r)^2]\cap I.\]
\end{thm}

Now, we can prove (B) of Theorem \ref{T101}. For readers' convenience, we restate it.

\begin{thm}
\label{thm:canB}
For any positive constants $A$ and $\ep$, there exists a positive constant $\sigma=\sigma(\ep,A)$ satisfying the following property.

Let $(M^2,g,J,f)$ be a K\"ahler Ricci shrinker surface with $\boldsymbol{\mu}(g) \ge -A$. If $R(x)+R^{-1}(x) \le \sigma f(x)$, then $(M,g,J,x)$ is $\ep$-close to the K\"ahler Ricci flow $( \CP^1 \times \C ,g_c(t),J_c) $.
\end{thm}

\begin{proof}
For fixed $A$ and $\bar{\ep}$, if no such $\sigma$ exists, one can obtain a sequence of K\"ahler Ricci shrinker surfaces $(M^2_i, g_i,J_i,f_i) \in \mathcal M(A)$ with $x_i \in M_i$ satisfying
\begin{align}\label{E404}
\lim_{i \to \infty} \frac{1}{R_{g_i}(x_i) f_i(x_i)}+\frac{R_{g_i}(x_i)}{f_i(x_i)} =0,
\end{align}
but $(M_i,g_i,J_i,x_i)$ is not $\bar{\ep}$-close to $( \CP^1 \times \C ,g_c(t))$. We set $s_i:=f_i(x_i)$, then it is clear from \eqref{E404} that
\begin{align*}
\lim_{i \to \infty} s_i=+\infty.
\end{align*}

Now, we define
\begin{align}\label{E405}
r_i:= \sup\{ r>0 \mid R_{g_i}(x_i,t) < r^{-2} \quad \text{for any} \quad t\in [-r^2,0]\}.
\end{align}

\textbf{Claim 1:} For $r_i$ defined in \eqref{E405}, we have
\begin{align}\label{E407}
\lim_{i \to \infty} s_i^{-1} (r_i^{2}+r_i^{-2})=0.
\end{align}

\emph{Proof of Claim 1}: Since $R_{g_i}(x_i)s_i\le r_i^{-2} s_i$, it is clear from \eqref{E404} that $s_i^{-1}r_i^2 \to 0$. Therefore, we only need to prove $s_i^{-1}r_i^{-2} \to 0$.

For simplicity, we set $x^t:=\psi_i^t(x)$ for any $x \in M_i$, where $\psi_i^t$ is the family of diffeomorphisms in \eqref{E201a}. If \eqref{E407} fails, we assume, by taking a subsequence,
\begin{align}\label{E407xa}
\lim_{i \to \infty} s_i^{-1}r_i^{-2}=\delta \in (0,\infty].
\end{align}

In particular, since $s_i \to +\infty$, \eqref{E407xa} implies that $r_i \to 0$. From the definition \eqref{E405}, there exists a $\tau_i \in [0,r_i^2]$ such that $R_{g_i}(x_i,-\tau_i) =r_i^{-2}$. For any $t \in [-\tau_i,0]$, we have
\begin{align*}
0 \le \frac{df_i(x_i^t)}{dt}=\frac{|\na_{g_i} f_i|_{g_i}^2}{1-t}(x_i^t) \le \frac{f_i(x_i^t)}{1-t}.
\end{align*}
By solving the ODE, it implies that for $t \in [-\tau_i,0]$,
\begin{align}\label{E407c}
\frac{s_i}{1-t} \le f_i(x_i^t) \le s_i.
\end{align}

Now, we estimate the distance between $x_i$ and $x_i^{-\tau_i}$ by considering the path $x_i^t$ for $t \in [-\tau_i,0]$. More precisely,
\begin{align}\label{E407d}
d_{g_i}(x_i,x_i^{-\tau_i}) \le \int_{-\tau_i}^0 |\partial_t x_i^t|_{g_i} \,dt =\int_{-\tau_i}^0 \frac{|\na_{g_i} f_i|_{g_i}}{1-t}(x_i^t) \,dt \le \int_{-\tau_i}^0 \sqrt{f_i(x_i^t)} \,dt \le s^{\frac 1 2}_i \tau_i \le s^{\frac 1 2}_i r_i^2.
\end{align}

On the other hand, we have
\begin{align}\label{E407dx}
R_{g_i}(x_i^{-\tau_i})=(1+\tau_i)R_{g_i}(x_i,-\tau_i)=(1+\tau_i)r_i^{-2}.
\end{align}

Combining \eqref{E407xa}, \eqref{E407c}, \eqref{E407dx} and the fact that $R_{g_i}(x_i^{-\tau_i}) \le f_i(x_i^{-\tau_i})$, we conclude that 
\begin{align}\label{E407xb}
1 \ge \frac{R_{g_i}(x_i^{-\tau_i})}{f_i(x_i^{-\tau_i})} \ge \frac{(1+\tau_i)r_i^{-2}}{s_i} \ge \frac{\delta}{2},
\end{align}
if $i$ is sufficiently large. In particular, we have $\delta \in (0,2]$. Moreover, it follows from \eqref{E407d} that
\begin{align}\label{E407e}
d_{g_i}(x_i,x_i^{-\tau_i}) \le 2 \delta^{-1} s_i^{-\frac 1 2}
\end{align}
if $i$ is sufficiently large. In addition, it follows from \eqref{E407c} that if $i$ is sufficiently large,
\begin{align}\label{E407ex}
s_i \ge f_i(x_i^{-\tau_i}) \ge \frac{s_i}{1+\tau_i} \ge \frac{s_i}{1+r_i^2} \ge \frac{s_i}{1+2\delta^{-1}s_i^{-1}} \ge \frac{s_i}{2}.
\end{align}

Consequently, it follows from \eqref{E407xb}, \eqref{E407e}, \eqref{E407ex} and Theorem \ref{T101} (A) that 
\begin{align}\label{E407f}
R_{g_i}(x_i) \ge c_1 s_i
\end{align}
for some constant $c_1=c_1(A,\delta)>0$. However, \eqref{E407f} contradicts \eqref{E404}, and Claim 1 is proved.

Next, we define for $t \in I:=(-\infty,0]$
\begin{align}\label{E408}
\tilde g_i(t):=r_i^{-2}g_i(r_i^{2}t) \quad \text{and} \quad \tilde f_i(t):=s_i^{-\frac 1 2} r_i^{-1}(F_i(r_i^2t)-s_i),
\end{align}
where the function $F_i(t)$ is defined above \eqref{E201aa}.

Now, we regard the pointed Ricci flow $(M_i, \tilde g_i(t),x_i)_{t \in I}$ as a metric flow pair $(\XX^i,(\mu^i_t)_{t \in I})$, as described in Example \ref{exmp:ricci}. Then it follows from Theorem \ref{Fcom} that there is a correspondence $\CF$ between the metric flows $\XX^i$, $i \in \IN \cup \{ \infty \}$, such that on compact time-intervals
\begin{equation} \label{E409}
(\XX^i, (\mu^i_t)_{t \in I}) \xrightarrow[i \to \infty]{\quad \IF, \CF \quad} (\XX^\infty, (\mu^\infty_t)_{t \in I}).
\end{equation}
Moreover, it follows from Theorem \ref{Thm_limit} that $\XX^\infty$ admits a decomposition
\[ \XX^{\infty}_{0}=\{x_{\infty}\},\quad \XX^{\infty}_{t<0}=\mathcal R \sqcup \mathcal S, \]
such that $\mathcal R$ is given by an $4$-dimensional Ricci flow spacetime $(\mathcal R, \tf, \partial^{\infty}_{\tf}, g^{\infty})$.

\textbf{Claim 2}: $\tilde f_i$ converges smoothly to a function $f_{\infty}$ on $\RR$ such that
\begin{equation} \label{E410}
\partial^\infty_{\tf} f_{\infty}\equiv 0, \quad \text{Hess}_{g^\infty} f_{\infty}\equiv 0, \quad \text{and} \quad |\na_{g^\infty} f_{\infty}|^2 \equiv 1. 
\end{equation}
\emph{Proof of Claim 2}: From the definitions \eqref{E408}, we compute from \eqref{E201aa}-\eqref{E201ac} that
\begin{align}
\partial_t \tilde f_i(t)&=-s_i^{-\frac 1 2} r_i(1-r_i^2t)R_{g_i(r_i^2 t)}=-s_i^{-\frac 1 2} r_i^{-1}(1-r_i^2t)R_{\tilde g_i(t)}, \label{E410a} \\
\text{Hess}_{\tilde g_i(t)} \tilde f_i(t) &= s_i^{-\frac 1 2} r_i^{-1} \lc \frac{g_i(r_i^2t)}{2}-(1-r_i^2t)\Rc_{g_i(r_i^2t)} \rc =\frac{1}{2} s_i^{-\frac 1 2} r_i \tilde g_i(t)-s_i^{-\frac 1 2} r_i^{-1}(1-r_i^2t) \Rc_{\tilde g_i(t)}, \label{E410b} \\
|\na_{\tilde g_i(t)} \tilde f_i|^2_{\tilde g_i(t)} &=s_i^{-1} \lc F_i(r_i^2 t)-(1-r_i^2 t)^2 R_{g_i(r_i^2 t)} \rc=1+s_i^{-\frac 1 2} r_i \tilde f_i(t)-s_i^{-1} r_i^{-2} (1-r_i^2t)^2 R_{\tilde g_i(t)}. \label{E410c}
\end{align}

Notice that by the definition of $r_i$ in \eqref{E405}, $R_{\tilde g_i}(x_i,t) \in [0,1]$ for any $t \in [-1,0]$. Therefore, it follows from Proposition \ref{prop:scalar} that there exists a constant $H=H(A) \ge H_4$ such that $(x_i,t)$ is an $H$-center of $(x_i,0)$ for any $t \in [-1,0]$.

Now, we fix a small constant $\delta>0$ and choose a time $-\tau \in (-\delta/2,0)$ such that the convergence \eqref{E409} is uniform at $t=-\tau$. Notice that the existence of $-\tau$ is 
guaranteed by Theorem \ref{Fcom}. Next, we fix a point $y \in \RR_{-\tau}$ and define $y_i \in M_i$ by $\phi_i(y)=(y_i,-\tau)$, where the maps $\phi_i$ are given by Theorem \ref{Thm_limit}(4).

It is clear that there exists a constant $L_1>0$ such that for sufficiently large $i$,
\begin{equation} \label{E411}
d_{\tilde g_i(-\tau)} (x_i,y_i) \le L_1.
\end{equation}
Otherwise, if $d_{\tilde g_i(-\tau)} (x_i,y_i) \to \infty$ by taking a subsequence, then it follows from Proposition \ref{prop:conc} that $\mu^i_{-\tau}(B_{\tilde g_i(-\tau)}(y_i,1)) \to 0$, since $(x_i,-\tau)$ is an $H$-center of $(x_i,0)$. 
However, it contradicts Theorem \ref{Thm_limit}(4)(a) since the conjugate heat kernel $v^i_{-\tau}$ of $\mu^i_{-\tau}$ converges smoothly to $v^\infty_{-\tau}$ of $\mu^\infty_{-\tau}$.

From \eqref{E411}, there exists a constant $L_2>0$ such that for sufficiently large $i$,
\begin{equation} \label{E412}
R_{\tilde g_i}(y_i,-\tau)+|\tilde f_i(y_i,-\tau)| \le L_2.
\end{equation}
Indeed, the fact that $R_{\tilde g_i}(y_i,-\tau)$ is uniformly bounded follows immediately from Theorem \ref{Thm_limit}(4), since $R_{\tilde g_i}(y_i,-\tau) \to R_{g^\infty}(y)$. Since $\tilde f_i(x_i,0)=0$, we derive from \eqref{E410a} and \eqref{E407} that
\begin{equation} \label{E412a}
|\tilde f_i(x_i,-\tau)| \le s_i^{-\frac 1 2} (r_i^{-1}+r_i \tau) \to 0.
\end{equation}
In addition, it follows from \eqref{E410c} that
\begin{equation} \label{E412b}
|\na_{\tilde g_i(-\tau)} \tilde f_i(-\tau)|^2_{\tilde g_i(-\tau)} \le 1+s_i^{-\frac 1 2} r_i \tilde f_i(-\tau)
\end{equation}
Combining \eqref{E411}, \eqref{E412a} and \eqref{E412b}, it is easy to see that $\tilde f_i(y_i,-\tau)$ is uniformly bounded.

Next, we consider any compact set $K \subset \RR_{t<-\delta}$. Since $\RR$ is a path-connected Ricci flow spacetime, one can construct a family of smooth curves $\gamma_z(t) \in \RR$ for $z \in K$ and $t \in [\tf(z),-\tau]$ such that 
\begin{enumerate}[label=\textnormal{(\alph{*})}]
\item $\gamma_z(t) \in \RR_t$ for any $t \in [\tf(z),-\tau]$ with $\gamma_z(\tf(z))=z$ and $\gamma_z(-\tau)=y$.
\item $\bigcup_{z \in K} \bigcup_{t \in [\tf(z),-\tau]} \{\gamma_z(t)\}$ is contained in a compact set $K' \subset \RR$.
\item $\sup_{z \in K}\sup_{t \in [\tf(z),-\tau]} |\gamma_z'(t)\vert_{\ker (d \mathfrak{t} )}|_{g^\infty_t} \le L_3$.
\end{enumerate}

We define $K_i=\phi_i(K)$ and $K'_i=\phi_i(K')$ for sufficiently large $i$. It follows immediately from Theorem \ref{Thm_limit}(4)(a) that there exists a constant $L_4>0$ such that
\begin{equation} \label{E412c}
R_{\tilde g_i} \le L_4
\end{equation}
on $K_i'$. For any $z \in K$, $\{(\gamma_z^i(t),t):=\phi_i(\gamma_z(t)) \} \subset K_i'$ is a spacetime curve connecting $(y_i,-\tau)$ with $\phi_i(z)$, for any $z \in K_i$. We compute directly from \eqref{E410a} and \eqref{E410c} that
\begin{equation} \label{E412d}
\left| \frac{d}{dt} \tilde f_i(\gamma^i_z(t),t) \right| \le C \lc s_i^{-\frac 1 2} (r_i^{-1}+r_i)+ \sqrt{1+s_i^{-\frac 1 2} r_i \tilde f_i(\gamma^i_z(t),t)} \rc,
\end{equation}
where we have used (c) above, \eqref{E412c} and Theorem \ref{Thm_limit}(4)(a). Combining \eqref{E407}, \eqref{E412} and \eqref{E412d}, it is easy to see that on $K_i$, for some constant $L_5>0$,
\begin{equation} \label{E412e}
|\tilde f_i| \le L_5.
\end{equation}

In other words, we conclude from \eqref{E412e} that the pull-back function $\phi_i^*(\tilde f_i)$ is locally uniformly bounded on $\RR_{t<-\delta}$. 
Since the convergence \eqref{E409} is smooth on $\RR$ by Theorem \ref{Thm_limit}(4), we can easily derive from \eqref{E407}, \eqref{E410a}, \eqref{E410b} and \eqref{E410c}, 
by taking a diagonal subsequence, that $\phi_i^*(\tilde f_i)$ converges smoothly to a function $f_{\infty}$ on $\RR_{t<-\delta}$ such that
\begin{equation*} 
\partial^\infty_{\tf} f_{\infty}\equiv 0, \quad \text{Hess}_{g^\infty} f_{\infty}\equiv 0, \quad \text{and} \quad |\na_{g^\infty} f_{\infty}|^2 \equiv 1. 
\end{equation*}
Since $\delta$ is arbitrary, we conclude, by further taking a diagonal sequence, that $f_{\infty}$ is globally defined on $\RR$ and \eqref{E410} holds. In sum, Claim 2 is proved.

\textbf{Claim 3:} For the limit $\XX^\infty$ in \eqref{E409}, we have
\begin{equation} \label{E413}
\XX^\infty_{t<0}=\RR' \times \R
\end{equation}
where $\RR'$ is a $3$-dimensional Ricci flow spacetime.

\emph{Proof of Claim 3}: From \eqref{E410} and the fact that $(\XX^\infty_t,d_t)$ is a completion of $(\RR_t, g^\infty_t)$, it can be derived from \cite[Theorem 15.50]{Bam20c} that
\begin{equation*} 
\XX^\infty_{t<0}=\XX' \times \R,
\end{equation*}
where $\XX'$ is a metric flow over $I'=(-\infty,0)$ with $\XX'=\RR' \sqcup \MS'$ and $(\RR', \tf', \partial'_{\tf}, g')$ is a $3$-dimensional Ricci flow spacetime such that $\RR' \times \R=\RR$ and $\MS' \times \R=\MS$. To finish the proof of Claim 3, we only need to prove $\MS'=\emptyset$.

Suppose otherwise. We fix a point $(\bar z, \bar t) \in \MS'$. It follows from Theorem \ref{Thm_limit}(2) that any tangent flow $(M_{\infty},g_{\infty})$ at $(\bar z, \bar t)$ is isometric to standard $S^2 \times \R$, $\RP^2 \times \R$ or $S^2\times_{\Z_2} \R$, where $\Z_2$ acts on $\R$ by reflection and on $S^2$ by the antipodal map.
 From the definition of the tangent flow and \cite[Theorem 9.21]{Bam20b}, there exist a sequence of times $t_j \nearrow \bar t$ and points $z_j \in \RR'_{t_j}$ such that $\MS'_{t_j}=\emptyset$ and $(\RR'_{t_j},R_{g'}(z_j) g'_{t_j},z_j)$ converges smoothly to $(M_{\infty},g_{\infty})$. 
 Here, the choice of $t_j$ is possible thanks to $\text{dim}_{\MM^*}\MS' \le 1$ by Theorem \ref{Thm_limit}(1). Moreover, since $(M_{\infty},g_{\infty})$ has constant scalar curvature $1$ and the corresponding convergence is smooth, the scaling factor can be $R_{g'}(z_j)$.

If $(M_{\infty},g_{\infty}) = \RP^2 \times \R$, then, by the smooth convergence on $\RR$, there exists a smooth embedding $\RP^2 \times (-1,1)^2$ to $M_i$, if $i$ is sufficiently large. However, it contradicts the orientability of $M_i$.

If $(M_{\infty},g_{\infty}) = S^2 \times \R$ or $S^2\times_{\Z_2} \R$, in either case, for any $\ep>0$, we can choose $z_j' \in \RR'_{t_j}$ near $z_j$, such that $(\RR'_{t_j}, g'_{t_j},z'_j)$ is $\ep$-close to $S^2 \times \R$, if $j$ is sufficiently large. 

Now, it follows from the smooth convergence on $\RR$ that the complex structure $J_i$, through the pull-back of $\phi_i$, converges smoothly to a complex structure $J_{\infty}$ on $\RR$. Therefore, it is easy to see from \eqref{E410} that the vector field $W:=J_{\infty}(\na_{g^\infty} f_{\infty}) \in \ker(d\tf)$ satisfies
\begin{equation} \label{E413a}
\partial^\infty_{\tf} W \equiv \na_{g^\infty} W\equiv 0, \quad |W|_{g^\infty}\equiv1 ,\quad \text{and} \quad \la \na_{g^\infty} f_{\infty},W \ra_{g^{\infty}}\equiv 0.
\end{equation}

From \eqref{E413a}, $W$ can be regarded as a time-independent parallel vector field on $\RR'$ with $W \in \ker(d\tf') $. We denote the integral curve of $W$ on $\RR'_{t_j}$ by $\gamma(t)$ such that $\gamma(0)=z_j'$. Since $W$ is parallel and $(\RR'_{t_j}, g'_{t_j})$ is complete, it follows that $(\RR'_{t_j}, g'_{t_j},\gamma(t))$ is $\ep$-close to $S^2 \times \R$ for all $t \in \R$. 
Consequently, we can patch all parts together to form an $S^2$-fibration where the fibers are nearly totally geodesic (cf. \cite[Appendix]{MT07} \cite[Section 67]{KL08} \cite[Section 7.3]{CZ06}), and conclude that one of the following two cases holds; see Figure \ref{fig1} and Figure \ref{fig2}.

\begin{enumerate}[label=\textnormal{(\roman*)}]
\item $\RR'_{t_j}$ is diffeomorphic to $S^2 \times \R$.
\item $\RR'_{t_j}$ is diffeomorphic to $S^2 \times S^1$.
\end{enumerate}

\begin{figure}[H]
\centering
\includegraphics[scale=0.2]{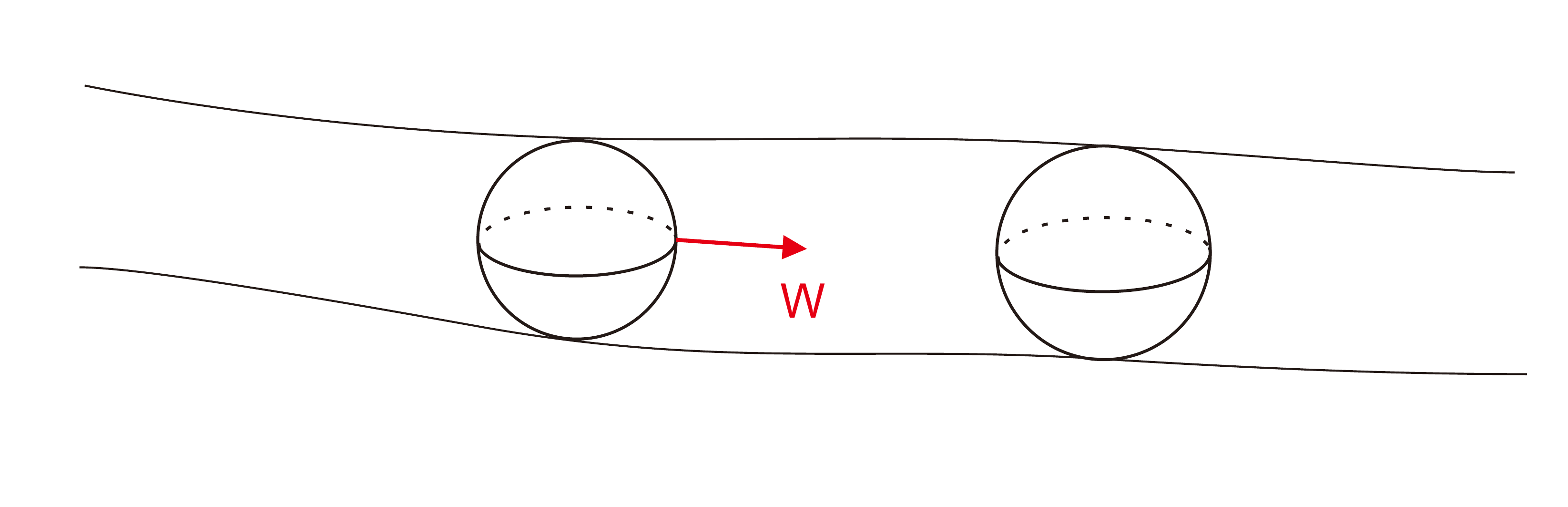}
\vspace*{-2mm}
\caption{$\RR'_{t_j}\cong S^2 \times \R$}
\label{fig1}
\end{figure}

\begin{figure}[H]
\centering
\includegraphics[scale=0.3]{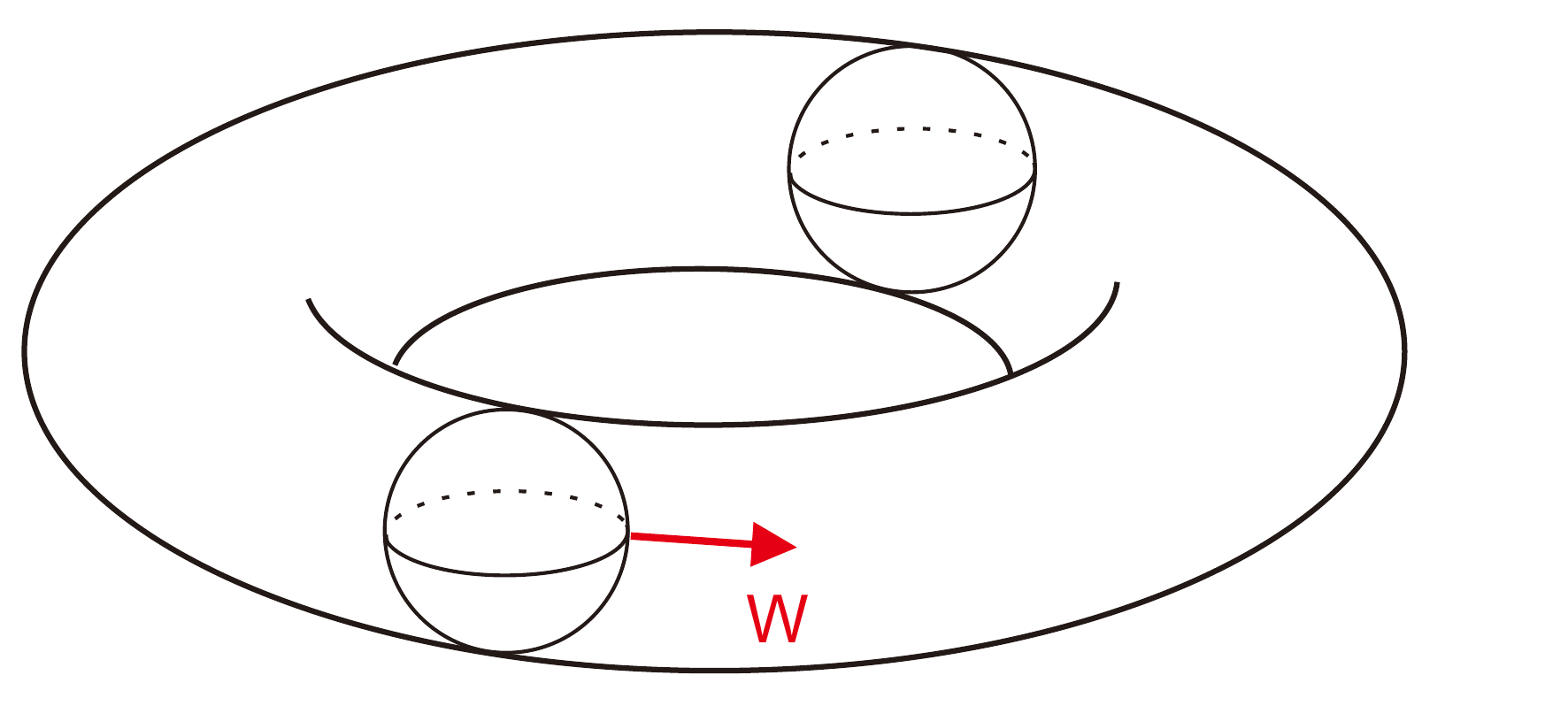}
\vspace*{+5mm}
\caption{$\RR'_{t_j}\cong S^2 \times S^1$}
\label{fig2}
\end{figure}

Notice that here $\RR'_{t_j}$ cannot be diffeomorphic to the non-orientable $S^2$-bundle over $S^1$ since otherwise, it contradicts the orientability of $M_i$ as above. In any case, $(\RR'_{t_j}, g'_{t_j},z)$ is $\ep$-close to $S^2 \times \R$, for any point $z\in \RR'_{t_j}$. 
Moreover, since $W$ is parallel, $|\Rm_{g'_{t_j}}|$ is almost a constant on $\RR'_{t_j}$ and hence $|\Rm_{g'_{t_j}}| \to +\infty$ everywhere uniformly, as $j \to \infty$. However, for a point $z' \in \RR'_{\bar t}$, we can find a sequence $z''_j \in \RR'_{t_j}$ converges to $z'$ in $\RR'$. Therefore, we obtain a contradiction and hence $\MS'= \emptyset$.
In sum, Claim 3 is proved.

\textbf{Claim 4:} The Ricci flow spacetime $\RR'$ is isometric to the conventional Ricci flow $(\R^3,g_E)$ or $(S^2 \times \R,g_c'(t))_{t <0}$, where $g_c'(t)$ is the standard product metric.

\emph{Proof of Claim 4}: From \eqref{E413}, we conclude that $(\RR'_t,g'_t)$ is complete for any $t<0$. Therefore, it follows from Proposition \ref{prop_3dRFST} that $(\RR'_t,g'_t)$ has nonnegative sectional curvature. 
Now, we only need to prove $\RR'$ is given by a conventional Ricci flow $(M_{\infty},g'(t))_{t \in I'}$ such that $(M_{\infty},g'(t))_{t \in I'}$ has bounded curvature on each compact time-intervals. 
Indeed, in this case, it follows from Theorem \ref{T201a} that $(M_{\infty} \times \R,g'(t) \times g_E,J_{\infty})_{t \in I'}$ is a $\kappa$-noncollapsed ancient solution to the K\"ahler Ricci flow with nonnegative bisectional curvature. 
Then we conclude from \cite[Theorem 1.3]{CL20} that $(M_{\infty} \times \R,g'(t) \times g_E,J_{\infty})_{t \in I'}$ is isometrically biholomorphic to $(\C^2,g_E)$ or $(\CP^1 \times \C,g_c(t))$. Therefore, Claim 4 follows immediately.

To prove that $\RR'$ is given by a conventional Ricci flow with bounded curvature on each compact time-interval, we only need to prove $(\RR'_t,g'_t)$ has uniformly bounded curvature by the two-sided pseudolocality Theorem \ref{thm:pseudo}.

Given $\bar t<0$, we suppose $(\RR'_{\bar t},g'_{\bar t})$ has unbounded curvature. By a well-known point-picking argument, there exists a sequence of points $z_j \in \RR'_{\bar t}$ with $Q_j:=R_{g'}(z_j) \to \infty$ such that for any $L>0$, $R_{g'}(y) \le 2Q_j$ for any $y \in B_{g'_{\bar t}}(z_j,2Q_j^{-\frac 1 2} L)$. 
From the two-sided pseudolocality theorem \ref{thm:pseudo}, there exists a small constant $\delta_0=\ep_0/2$ such that one obtains an embedding of $N_j \times (-2\delta_0 Q_j^{-1},2\delta_0 Q_j^{-1})$ to a neighborhood of $z_j$ in $\RR'$, on which $g'$ is given by the conventional Ricci flow, where $N_j=B_{g'_{\bar t}}(z_j,\delta_0 L Q_j^{-\frac 1 2})$. 
Next, we set $z'_j=(z_j,\delta_0 Q_j^{-1}) \in \RR'_{\bar t+\delta_0 Q_j^{-1}}$ under this identification. By shifting the time by $\bar t+\delta_0 Q_j^{-1}$ and rescaling parabolically $g'$ by $Q_i$, the metric flow pairs $\lc \RR'_{t \le \bar t+\delta_0 Q_j^{-1}}, (v_{z_j';t})_{t \le \bar t+\delta_0 Q_j^{-1}} \rc$,
 by taking a subsequence, $\IF$-converge to a metric flow pair $(\XX'',(\mu_t'')_{t \in I})$ with $\XX''=\RR''\sqcup \MS''$ and $(\RR'', \tf'', \partial''_{\tf}, g'')$ is a $3$-dimensional Ricci flow spacetime. Moreover, the parallel vector field $W$ also converges smoothly to a vector field $W''$ on $\RR''$ satisfying as \eqref{E413a}
\begin{equation} \label{E413aa}
\partial_{\tf''} W'' \equiv \na_{g''} W'' \equiv 0 \quad \text{and} \quad \quad |W''|_{g''}\equiv1.
\end{equation}

Therefore, by the same arguments as in the proof of Claim 3, we immediately conclude that $\MS''=\emptyset$. Since $(\RR',g')$ has nonnegative sectional curvature, $(\RR''_{-\delta_0},g''_{-\delta_0})$ splits a line from Toponogov’s splitting theorem. 
Since $(\RR''_{-\delta_0},g''_{-\delta_0})$ is nonflat by our construction, it is clear that the parallel vector field $W''$ is along the splitting direction. Therefore, it follows from \eqref{E413aa} that $W''$ induces a splitting direction for each $t\le -\delta_0$. In other words, we conclude that
\begin{equation*}
\RR''_{t \le -\delta_0}= \RR''' \times \R,
\end{equation*}
where $\RR'''$ is a $2$-dimensional Ricci flow spacetime defined on $(-\infty,-\delta_0]$. Now, we can repeat the above argument to show that $\RR'''$ is given by a conventional Ricci flow, which is a nonflat $\kappa$-noncollapsed ancient solution to the Ricci flow with nonnegative and bounded curvature on each compact time-interval. 
Therefore, it is known from \cite{Pe1} that such a solution must be isometric to standard $S^2$ or $\RP^2$. Notice the latter is excluded because of orientability as before. From the smooth convergence at $-\delta_0$ and the fact that the rescaling factor $Q_i \to +\infty$, we conclude that $(\RR'_{\bar t},g'_{\bar t})$ contains cylindrical $S^2 \times \R$-like regions with arbitrarily small radius. However, it contradicts \cite[Proposition 2.2]{CZ06a} (see also the proof of \cite[Theorem 46.1]{KL08}) since $(\RR'_{\bar t},g'_{\bar t})$ has nonnegative sectional curvature.

In sum, we have proved that $(\RR'_t,g'_t)$ has uniformly bounded curvature for any $t<0$. Therefore, Claim 4 is confirmed by the reasons mentioned above.

\textbf{Claim 5}: $(M_i,\tilde g_i(t),J_i,x_i)_{t \in I}$ converges smoothly to $(\CP^1 \times \C ,g_c(t) ,J_c)_{t \in I}$.

\emph{Proof of Claim 5:} From Claim 3 and Claim 4, we deduce that $\XX^\infty_{t<0}$ is a given by the conventional K\"ahler Ricci flow $(\C^2,g_E)$ or $(\CP^1 \times \C ,g_c(t))_{t<0}$. As proved before, $(x_i,\bar t)$ is an $H$-center of $(x_i,0)$ with respect to $\tilde g_i(t)$ for any $\bar t \in [-1,0)$. 
We choose $\bar t =-\ep_0$, where $\ep_0$ is a small constant determined later. It follows from Proposition \ref{prop_limit2} that for any $\bar t \in [-1,0)$,
\begin{equation} \label{E415}
(M_i,\tilde g_i(t), x_i)_{t \le \bar t} \longright{pointed-{C}^{\infty}-Cheeger-Gromov} (M_\infty, g_{\infty}(t))_{t \le \bar t},
\end{equation}
where $(M_\infty, g_{\infty}(t))=(\C^2,g_E)$ or $(\CP^1 \times \C ,g_c(t))$. In either case, since the convergence \eqref{E415} is smooth at $t=\bar t$ and $R_{\tilde g_i}(x_i,\bar t) \le 1$, it is clear that for any $L>1$, $|\Rm_{\tilde g_i}| \le 1.1$ on $B_{\tilde g_i(\bar t)}(x_i,L)$, if $i$ is sufficiently large. 
Therefore, it follows from Theorem \ref{thm:pseudo} and Shi's local estimates \cite{Shi89a} that
\begin{equation} \label{E415a}
(M_i,\tilde g_i(t), x_i)_{t \le 0 } \longright{pointed-{C}^{\infty}-Cheeger-Gromov} (M_\infty, g_{\infty}(t))_{t \le 0},
\end{equation}
if $\ep_0$ is sufficiently small. Since $R_{\tilde g_i}(x_i,-\tau_i)=1$ for some $\tau_i \in [0,1]$ by our definition \eqref{E405}, we immediately conclude from \eqref{E415a} that $(M_\infty, g_{\infty}(t))$ is nonflat. Therefore, $(M_\infty, g_{\infty}(t))=(\CP^1 \times \C ,g_c(t))$ by the uniqueness of the Ricci flow \cite{CZ06}. 
Moreover, since the complex structure compatible with $g_c$ is unique, $J_i$ converges smoothly to $J_c$, and hence Claim 5 is proved.

In summary, we have proved 
\begin{equation*}
(M_i,\tilde g_i(t), J_i,x_i)_{t \le 0} \longright{pointed-{C}^{\infty}-Cheeger-Gromov} (\CP^1 \times \C ,J_c, g_c(t))_{t \le 0}.
\end{equation*}
Therefore, we obtain a contradiction, and the proof of the theorem is complete.
\end{proof}

\section{Proof of the main theorem}

This section aims to prove Theorem \ref{T102}. Obviously, we only need to consider noncompact K\"ahler Ricci shrinkers. We first recall the following result from \cite[Theorem 0.1]{MW15}.

\begin{thm}\label{thm:oneend}
Let $(M^m,g,J,f)$ be a noncompact K\"ahler Ricci shrinker. Then $(M,g)$ has only one end.
\end{thm}

Next, for any Ricci shrinker, we have the following estimate of the integral of the scalar curvature; see \cite[Remark 3.1]{CZ10}.

\begin{lem}\label{lem:ave}
Let $(M^n,g,f)$ be a Ricci shrinker. Then for any $s>0$,
\begin{equation} \label{E501}
\aint_{\Sigma(0,s)} R \,dV_g := \frac{1}{|\Sigma(0,s)|}\int_{\Sigma(0,s)} R \,dV_g \le \frac{n}{2}.
\end{equation}
\end{lem}

\begin{proof}
By taking the trace of \eqref{E100}, we have $R+\Delta f=n/2$. Integrating this identity on $\Sigma(0,s)$ and performing integration by parts, we obtain
\begin{align*}
\int_{\Sigma(0,s)} R \,dV=&\frac{n}{2}|\Sigma(0,s)|-\int_{\Sigma(0,s)} \Delta f \,dV \\
=&\frac{n}{2}|\Sigma(0,s)|-\int_{\Sigma(s)} |\na f| \,d\sigma \le \frac{n}{2}|\Sigma(0,s)|,
\end{align*}
if $s$ is a regular value of $f$. In general, \eqref{E501} follows by approximation.
\end{proof}

In addition, we have the following linear growth estimate of volume lower bound (cf. \cite[Theorem 1.6]{MW12} and \cite[Proposition 6]{LW20}). 
Notice that by Lemma \ref{L201}, $B(p,c_1\sqrt s) \subset \Sigma(0,s) \subset B(p,c_2\sqrt s)$ if $s \ge c_3$, for some positive constants $c_i=c_i(n)$ for $i=1,2,3$.

\begin{prop}\label{prop:volume}
For any noncompact Ricci shrinker $(M^n, g, f) \in \MM(A)$, there exist positive constants $s_1 = s_1(n)$ and $\ep_1$ =$\ep_1(n,A)$ such that for any $s \ge s_1$,
\begin{equation} \label{E502}
|\Sigma(0,s)| \ge \ep_1 \sqrt{s}.
\end{equation}
\end{prop}

Combining \eqref{E501} and \eqref{E502}, we know that on any noncompact Ricci shrinker, the scalar curvature cannot exceed $n$ everywhere outside a compact set.

We proceed to prove Theorem \ref{T102}, which we restate here.

\begin{thm}\label{thm:bound}
Any noncompact K\"ahler Ricci shrinker surface $(M^2,g,J,f)$ has bounded scalar curvature.
\end{thm}

\begin{proof} 
We assume $(M,g,f) \in \MM( \bar A)$ for some constant $ \bar A>0$. We fix a small constant $\ep>0$ to be determined later, and set 
\begin{align}
\bar \sigma:=\min\{\sigma(\ep, \bar A),1/1000\}, 
\quad \bar L:=\max\{L(\ep, \bar \sigma/2, \bar A),1000 \bar \sigma^{-1}\} , \label{eqn:barL}
\end{align}
where $\sigma$ and $L$ are functions in Theorem \ref{T101}.

Now, we define the following subsets of $\Sigma(\bar L,\infty)$:
\begin{align*}
\begin{cases}
&A \coloneqq \left \{x\in \Sigma(\bar L,\infty) \mid R(x) \ge \bar \sigma f(x)/2 \right \}, \\
&B \coloneqq \left\{x\in \Sigma(\bar L,\infty) \mid R(x) < \bar \sigma f(x)/2 \quad \text{and} \quad R^{-1}(x)<\bar \sigma f(x)/2 \right\},\\
&C \coloneqq \left\{x\in \Sigma(\bar L,\infty) \mid R^{-1}(x) \ge \bar \sigma f(x)/2\right \}.
\end{cases}
\end{align*}
It is clear that $\Sigma(\bar L,\infty)=A\sqcup B \sqcup C$. Moreover, it follows from Theorem \ref{T101} that for any $x \in A$, $(M,g,x)$ is $\ep$-close to a pointed K\"ahler steady soliton orbifold. On the other hand, for any $x \in B$, $(M,g,J,x)$ is $\ep$-close to $( \CP^1 \times \C ,g_c(t),J_c) $.

For any $x \in M$, we set $x^t:=\psi^t(x)$ for $t<1$, where $\psi^t$ is the family of diffeomorphisms in \eqref{E201a}. We fix a small constant $0<\delta \ll 1$ and let $\ep$ depend on it. Moreover, we set $\eta_{\ep}$ to be a positive function such that $\eta_{\ep} \to 0$ if $\ep \to 0$. In the following, $\eta_{\ep}$ may be different line by line.

\textbf{Claim 1}: There exists a constant $\bar t=\bar t(\delta) \in (0,1)$ such that for any $x \in B$ with $R(x) \ge 1+\delta$, there exists a constant $t_1 \le \bar t$ such that $x^t \in B$ for any $t \in [0,t_1)$ and $x^{t_1} \in A$.

\emph{Proof of Claim 1}: Since $B$ is an open set, there exists a largest constant $t_1 \in (0,1]$ such that $x^t \in B$ for any $t \in [0,t_1)$. By our definition, $(M,g,J,x^t)$ is $\ep$-close to $( \CP^1 \times \C ,g_c(t),J_c) $ for any $t \in [0,t_1)$. Since the Ricci flow associated with the Ricci shrinker is self-similar, we conclude that $(M,g(t),J,(x,t))$ is also $\ep$-close to $( \CP^1 \times \C ,g_c(t),J_c) $ for any $t \in [0,t_1)$.

Therefore, the evolution equation of the scalar curvature satisfies
\begin{equation} \label{E503a}
(1-\eta_{\ep}) R^2\le \partial_t R=\Delta_t R+2|\Rc|^2 \le (1+\eta_{\ep}) R^2.
\end{equation}
for any $(x,t)$ with $t \in [0,t_1)$. Indeed, \eqref{E503a} holds, as $\partial_t R/R^2$ is invariant under parabolic scaling, and the model space satisfies \eqref{eq:scalar}. In particular, we have
\begin{align*}
1-\eta_{\ep} \leq \partial_t \left(-\frac{1}{R} \right) \leq 1+\eta_{\ep}.
\end{align*}
Integrating the above inequality over time, we conclude that if $\ep$ is sufficiently small,
\begin{equation} \label{E503b}
R(x,t) \ge \frac{R(x)}{1-(1-\eta_{\ep})R(x)t}\ge \frac{1+\delta}{1-(1-\eta_{\ep})(1+\delta)t},
\end{equation}
for any $t \in [0,t_1)$. From \eqref{E503b}, it is clear that $t_1<1/(1-\eta_{\ep})(1+\delta)<1/(1+\delta/2):= \bar t$ for sufficiently small $\ep$. In addition, since $R(x,t)=(1-t)^{-1}R(x^t)$ and $R(x) \geq 1+\delta \geq (1-\eta_{\ep})^{-1}$ for $\ep$ sufficiently small, we obtain from \eqref{E503b} that
\begin{equation} \label{E503c}
R(x^t) \ge \frac{(1-t)R(x)}{1-(1-\eta_{\ep})R(x)t} \ge R(x)
\end{equation}
for any $t \in [0,t_1)$. 
As $f(x^t)$ is increasing, it follows from \eqref{E503c} that for any $t \in [0,t_1)$,
\begin{equation*}
R(x^t) f(x^t) \ge R(x)f(x) >2\bar{\sigma}^{-1},
\end{equation*}
where the last inequality holds because $x \in B$. However, by the definition of $t_1$, $x^{t_1}$ is on the boundary of $B$.
Therefore, either $R(x^{t_1})=\bar \sigma f(x^{t_1})/2$, or $R(x^{t_1})f(x^{t_1})=2\bar{\sigma}^{-1}$, which is excluded by the above inequality. 
Consequently, $R(x^{t_1})=\bar \sigma f(x^{t_1})/2$ and $x^{t_1} \in A$. The proof of Claim 1 is complete. 

\textbf{Claim 2}: For any $x \in \Sigma(\bar L,\infty)$ with $R(x) < 1-\delta$, we have $R(x^t) < 1-\delta$ for any $t \in [0,1)$.

\emph{Proof of Claim 2}: Suppose the claim fails. There exists a constant $t_2 \in (0,1)$ such that $R(x^t)<1-\delta$ for any $t \in [0,t_2)$ and $R(y)=1-\delta$ for $y:=x^{t_2}$.

Note that $y \in B$. Actually, if $y \in A$, then by the choice of $\bar{L}$ in (\ref{eqn:barL}), we have $1-\delta=R(y) \ge \bar \sigma \bar{L}/2 \ge 500$, which is absurd. 
Similarly, if $y\in C$, then $1/2<1-\delta=R(y)<2/\bar{\sigma} \bar{L}<1/1000$, which is impossible.
Thus $y \in B$ and we can apply canonical neighborhood theorem and \eqref{E503a} to obtain
\begin{equation} \label{E504a}
R(y,t) \ge \frac{R(y)}{1-(1+\eta_{\ep})R(y)t} = \frac{1-\delta}{1-(1+\eta_{\ep})(1-\delta)t}
\end{equation}
for any $t \in (-\ep^{-1} R^{-1}(y),0]$. Therefore, it follows from \eqref{E504a} that
\begin{equation*}
R(y^t) \ge \frac{(1-t)(1-\delta)}{1-(1+\eta_{\ep})(1-\delta)t} > 1-\delta
\end{equation*}
for any $t \in (-\ep^{-1} R^{-1}(y),0)$, if $\ep$ is sufficiently small. However, this contradicts the fact that $R(x^t)<1-\delta$ for any $t \in [0,t_2)$. Therefore, Claim 2 is proved.

\textbf{Claim 3}: For any $x \in A$, we have $R(x^t) \ge \bar \sigma f(x^t)/4$ for any $t \in [0,1)$.

\emph{Proof of Claim 3}: Suppose otherwise, there exists a constant $t_3 \in (0,1)$ such that $R(x^t) > \bar \sigma f(x^t)/4$ for any $t \in [0,t_3)$ and $R(y) = \bar \sigma f(y)/4$, where $y:=x^{t_3}$. In particular, $y \in B$ by our choice of $\bar L$ and hence \eqref{E503a} holds for any $(y,t)$ for $t \in (-\ep^{-1} R^{-1}(y),0]$. Therefore, we conclude
\begin{equation} \label{E505a}
R(y,t) \le \frac{R(y)}{1-(1-\eta_{\ep})R(y)t}
\end{equation}
for any $t \in (-\ep^{-1} R^{-1}(y),0]$. We consider the following two cases.

(a). $R^{-1}(y) < t_4:=\dfrac{t_3}{1-t_3}$.

In this case, $y^{-R^{-1}(y)}=x^{t'}$ for some $t' \in (0,t_3)$. Indeed, this follows from the fact that
\begin{equation*} 
\psi^s(\psi^t(z))=\psi^t(\psi^s(z))=\psi^{s+t-st}(z)
\end{equation*}
for any $z\in M$ and $s,t <1$, by the definition \eqref{E201a}. Moreover, we compute from
\begin{equation*}
\frac{d}{dt}f(x^t)=\frac{|\na f|^2}{1-t}(x^t) \le \frac{f(x^t)}{1-t}
\end{equation*}
that 
\begin{equation} \label{E505e}
f(x^{t'})=f(y^{-R^{-1}(y)}) \ge \frac{f(y)}{1+R^{-1}(y)}.
\end{equation}
In addition,
\begin{equation} \label{E505ex}
R(x^{t'})=R(y^{-R^{-1}(y)})=(1+R^{-1}(y)) R(y, -R^{-1}(y)).
\end{equation}

Combining \eqref{E505a}, \eqref{E505e} and \eqref{E505ex}, we conclude
\begin{equation} \label{E505f}
\frac{R(x^{t'})}{f(x^{t'})} \le \frac{(1+R^{-1}(y))^2}{2-\eta_{\ep}} \frac{R(y)}{f(y)} = \frac{\bar \sigma (1+R^{-1}(y))^2}{4(2-\eta_{\ep})}.
\end{equation}
Since $R^{-1}(y) \ll 1$ by our choice, \eqref{E505f} contradicts the fact that $R(x^{t'}) > \bar \sigma f(x^{t'})/4$.

(b). $R^{-1}(y) \ge t_4$.

In this case, $y^{-t_4}=x$ and similar to \eqref{E505f} we have
\begin{equation} \label{E505g}
\frac{R(x)}{f(x)} \le \frac{(1+t_4)^2}{1+(1-\eta_{\ep})R(y)t_4} \frac{R(y)}{f(y)} \le \frac{\bar \sigma (1+t_4)^2}{4}.
\end{equation}
Since $t_4 \le R^{-1}(y) \ll 1$, \eqref{E505g} contradicts the fact that $R(x) \ge \bar \sigma f(x)/2$. Therefore, we have proved Claim 3.

From Sard's theorem, there exists $s_0 > \bar L$ such that $\Sigma(s_0)$ is a smooth hypersurface. We assume
\begin{equation} \label{eq:comp}
\Sigma(s_0)=\Sigma_1 \sqcup \cdots \sqcup \Sigma_k,
\end{equation}
where each $\Sigma_i$ is a connected component of $\Sigma(s_0)$ for $1 \le i \le k$. In addition, we set $\phi^s$ to be a family of diffeomorphisms generated by $\na f/|\na f|^2$ such that $\phi^{s_0}=\text{id}$. 
Clearly, $\phi^s$ is just a reparametrization of $\psi^t$. Here, if $\phi^s(z)$ converges to a critical point $z_0$ of $f$ as $s \nearrow \bar s$ (or $s \searrow \bar s$), we define $\phi^s(z)=z_0$ for any $s \ge \bar s$ (or $s \le \bar s$). Moreover, we set $\Sigma_i^s:=\phi^s(\Sigma_i)$. 

For each $i$, we consider the following four cases.

\textbf{Case (i)}: There exists a point $x \in \Sigma_i$ with $R(x) \le 1-2\delta$.

In this case, we claim that $R(z) < 1- \delta$ for all $z \in \Sigma_i$. Suppose otherwise. There exists a point $y \in \Sigma_i$ with $R(y)=1-\delta$. For the same reason as in the proof of Claim 2 above, $y \in B$. 
From Theorem \ref{T101}(B), $(M,g,J,y)$ is $\ep$-close to $( \CP^1 \times \C ,g_c(t),J_c)$. Moreover, it follows from the proof of Theorem \ref{T101}(B) that $\{\na f,J(\na f) \}$ almost forms the splitting factor $\C$. Therefore, locally around $y$, $\Sigma_i$ is almost isometric to the standard $S^2 \times \R$.
In other words, $(\Sigma_i,g\vert_{\Sigma_i},y)$ is $\eta_{\ep}$-close to $S^2 \times \R$. Next, we denote the integral curve of $J(\na f)$ on $\Sigma_i$ by $\gamma(t)$ such that $\gamma(0)=y$. 
Since $J(\na f)$ is a Killing field, we conclude that $(\Sigma_i,g\vert_{\Sigma_i}, \gamma(t))$ is $\eta_{\ep}$-close to $S^2 \times \R$ for all $t \in \R$. 
Argued as in the proof of Theorem \ref{thm:canB} Claim 3, we patch all parts together to conclude that $\Sigma_i$ is diffeomorphic to $S^2 \times S^1$.

Moreover, $(\Sigma_i, g\vert_{\Sigma_i},z)$ is $\eta_{\ep}$-close to $S^2 \times \R$, for any point $z\in \Sigma_i$. Therefore, $R(z)>1-2\delta$ for any $z \in \Sigma_i$, if $\ep$ is sufficiently small. However, we obtain a contradiction by the existence of $x$.

It follows from Claim 2 that $R(z) <1-\delta$ for any $z \in \Sigma_i^s$ and $s \ge s_0$. Since $R+|\na f|^2=f$, $\Sigma_i^s$ contains no critical point of $f$ for any $s \ge s_0$. Thus, $\Sigma_i^s$ is diffeomorphic to $\Sigma_i$ and $\bigcup_{s \ge s_0} \Sigma_i^s$ forms an end of $M$. 
Since $(M,g)$ has only one end by Theorem \ref{thm:oneend}, we immediately conclude that $(M,g)$ has bounded scalar curvature. 

\textbf{Case (ii)}: For any $s \ge s_0$ and $z \in \Sigma_i^s$, $1-2\delta <R(z)<1+2\delta$.

In this case, it is evident that $\bigcup_{s \ge s_0} \Sigma_i^s$ forms an end on which the scalar curvature is bounded. 

\textbf{Case (iii)}: There exist a constant $s_{1,i}>s_0$ and a point $x \in \Sigma_i^{s_{1,i}}$ such that $1-2\delta <R(z)<1+2\delta$ for any $s \in [s_0,s_{1,i})$ and $z \in \Sigma_i^s$, and $R(x)=1+2\delta$.

In this case, $x \in B$ and hence by the same reason as in Case (i), $\Sigma^{s_{1,i}}_i$ is diffeomorphic to $S^2 \times S^1$ or $S^2 \times_{\Z_2} S^1$ such that $(\Sigma^{s_{1,i}}_i, g\vert_{\Sigma^{s_{1,i}}_i},z)$ is $\eta_{\ep}$-close to $S^2 \times \R$, for any point $z\in \Sigma^{s_{1,i}}_i$. 
Thus $R(z) \ge 1+\delta$ for any $z \in \Sigma^{s_{1,i}}_i$. Therefore, it follows from Claim 1 and Claim 3 that there exists a constant $s_{2,i}>s_{1,i}$ such that $R(\phi^s(z)) \ge \bar \sigma f(\phi^s(z))/4$ for any $z \in \Sigma_i$ and $s \ge s_{2,i}$.
Indeed, for any $z\in \Sigma^{s_{1,i}}_i$, it follows from Claim 1 that there exists a $t_z \in (0,\bar t]$ such that $z^t \in B$ for $t \in [0,t_z)$ and $z^{t_z} \in A$. We compute from
\begin{equation*}
\frac{d}{dt}f(z^t)=\frac{|\na f|^2}{1-t}(z^t) \le \frac{f(z^t)}{1-t}
\end{equation*}
that 
\begin{equation} \label{E506xa}
f(z^{t_z}) \le \frac{f(z)}{1-t_z} \le \frac{s_{1,i}}{1-\bar t}.
\end{equation}

Combining \eqref{E506xa} with Claim 3, we can choose $s_{2,i}:=\frac{s_{1,i}}{1-\bar t}$. With this choice, for any $z \in \Sigma_i$ and $s \ge s_{2,i}$, $\phi^s(z)=w^{t}$ for some $w \in \Sigma^{s_{1,i}}_i$ and $t \ge t_w$. Thus, we have $R(\phi^s(z)) \ge \bar \sigma f(\phi^s(z))/4$ for any $z \in \Sigma_i$ and $s \ge s_{2,i}$.

\textbf{Case (iv)}: There exists a point $x \in \Sigma_i$ such that $R(x) \ge 1+2\delta$.

In this case, similar to Case (i), we conclude that $R(z) \ge 1+\delta$ for any $z \in \Sigma_i$. For the same reason as in Case (iii), there exists a constant $s_{3,i}>s_0$ such that $R(\phi^s(z)) \ge \bar \sigma f(\phi^s(z))/4$ for any $z \in \Sigma_i$ and $s \ge s_{3,i}$.

We are ready to finish the proof according to the previous classification.
We prove this by contradiction. Suppose $M$ has unbounded scalar curvature. Namely, we assume 
\begin{align}
\sup_{M} R=\infty. \label{eqn:scalarinf}
\end{align}
Then Case (i) and Case (ii) will not appear, as in these cases, $M$ has an end with bounded scalar curvature. 
Since $M$ has only one end by Theorem~\ref{thm:oneend}, we see that $M$ has bounded scalar curvature, which contradicts~\ref{eqn:scalarinf}. 
Therefore, Case (iii) or Case (iv) above applies to each component $\Sigma_i$ for $1 \le i \le k$. By the analysis above, we conclude that there exists $\bar s:=\max_{1 \le i \le k}\{ s_{2,i},s_{3,i}\}$ such that
\begin{equation} \label{E507}
R(w) \ge \frac{\bar \sigma f(w)}{4}
\end{equation}
if $w=\phi^s(z)$ for some $z \in \Sigma(s_0)$ and $s \ge \bar s$. 
All of such $w$ form a subset of $\Sigma(\bar s,\infty)$. Although this subset may not equal $\Sigma(\bar s,\infty)$, 
we claim that \eqref{E507} still holds for any $w \in \Sigma(\bar s,\infty)$; see Figure \ref{fig3} below.

\begin{figure}[H]
\centering
\includegraphics[scale=0.33]{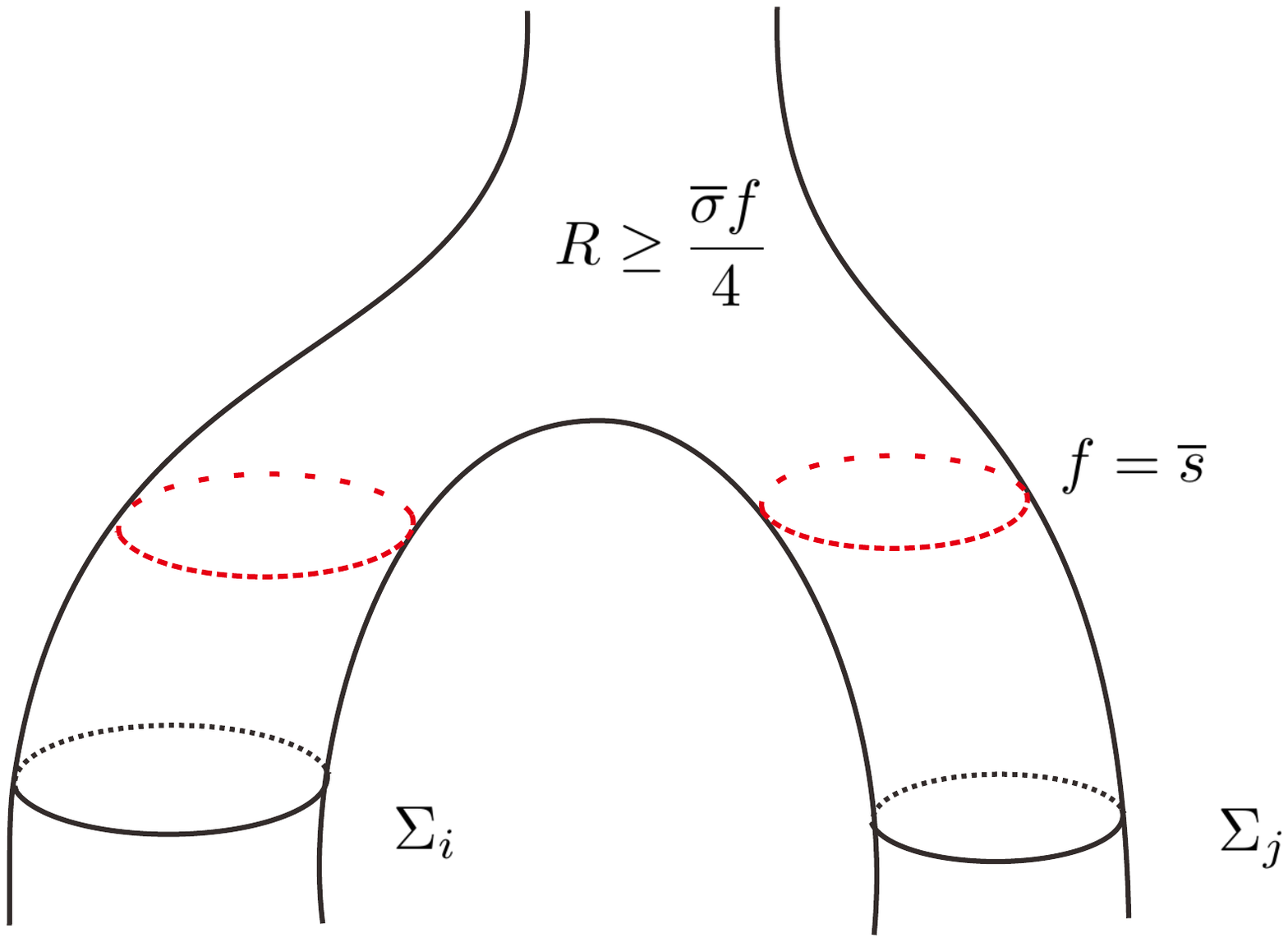}
\vspace*{+5mm}
\caption{Case (iii) or Case (iv) holds for all components}
\label{fig3}
\end{figure}

Indeed, for any $w \in \Sigma(\bar s,\infty) $, we consider $w_s:=\phi^{s-s_w+s_0}(w)$, where $s_w:=f(w) \ge \bar s$. If $w_s$ does not hit the critical point of $f$ for $s \in [s_0,s_w]$, then it is clear that $w=\phi^{s_w}(z)$ for some $z \in \Sigma(s_0)$ and hence \eqref{E507} holds. 
On the other hand, if $w_{s'}$ is a critical point of $f$ for some $s' \in [s_0,s_w]$, then \eqref{E507} also holds for $w$ by Claim 3 since $R(w_{s'})=f(w_{s'})$. 
Therefore, \eqref{E507} holds for any $w \in \Sigma(\bar s,\infty) $. 
Choose $\theta> \bar{s}$, it follows from \eqref{E507} that
\begin{align*}
\int_{\Sigma(0,\theta)} R \,dV \geq \int_{\Sigma(\bar{s},\theta)} R \,dV \geq \frac{\bar{\sigma}}{4} \bar{s} \cdot |\Sigma(\bar{s},\theta)|= \frac{\bar{\sigma}}{4} \bar{s} \cdot \left\{ |\Sigma(0,\theta)| -|\Sigma(0,\bar{s})| \right\}. 
\end{align*}
Recall that $\bar{s}>s_0>\bar{L}$ by our choice and $\bar{L}$ is sufficiently large by its definition (\ref{eqn:barL}). 
Taking the average of the above inequality and applying the volume linear growth estimate \eqref{E502}, we obtain
\begin{align*}
\liminf_{\theta \to \infty} \aint_{\Sigma(0,\theta)} R \,dV \geq \frac{\bar{\sigma} \bar{s}}{4} \geq \frac{\bar{\sigma} \bar{L}}{4} \gg 2=\frac{n}{2}, 
\end{align*}
which contradicts \eqref{E501}.
This contradiction indicates that the assumption (\ref{eqn:scalarinf}) is wrong. 
Therefore, we have proved that $(M,g)$ has bounded scalar curvature.

\end{proof}

\begin{rem}
As pointed out by the referee, the decomposition in \eqref{eq:comp} has only one connected component. This holds because, for any K\"ahler Ricci shrinker, the potential function $f$ is a Morse-Bott function with even indices (see \cite{Fra59}), implying that each of its level sets is connected. Therefore, the proof can be simplified somewhat. Nevertheless, we believe our argument may still be useful for general real Ricci shrinkers.
\end{rem}

\newpage

\appendixpage
\addappheadtotoc
\appendix
\section{Metric flows and $\mathbb F$-convergence} 
\label{app:A}

In this appendix, we state some theorems regarding the metric flows and $\IF$-convergence. The theory of $\IF$-convergence was originally established by Bamler for compact Ricci flows \cite{Bam20a, Bam20b, Bam20c}. 
In~\cite{LW23}, the theory is generalized by the authors to Ricci flows induced by Ricci shrinkers. For more detailed discussions, see~\cite[Section 6]{LW23}.

We first recall the following definition of the metric flow from~\cite[Definition 3.1]{Bam20b}.

\begin{defn}[Metric flow] \label{def:mf}
Let $I \subset \R$ be a subset.
A metric flow over $I$ is a tuple of the form 
\begin{equation*}
(\XX , \tf, (d_t)_{t \in I} , (v_{x;s})_{x \in \XX, s \in I, s \leq \tf (x)}) 
\end{equation*}
with the following properties:
\begin{enumerate}[label=(\arabic*)]
\item $\XX$ is a set consisting of points.

\item $\tf : \XX \to I$ is a map called time-function.
Its level sets $\XX_t := \tf^{-1} (t)$ are called time-slices and the preimages $\XX_{I'} := \tf^{-1} (I')$, $I' \subset I$, are called time-slabs.
\item $(\XX_t, d_t)$ is a complete and separable metric space for all $t \in I$.

\item $v_{x;s}$ is a probability measure on $\XX_s$ for all $x \in \XX$, $s \in I$, $s \leq \tf (x)$. For any $x \in \XX$ the family $(v_{x;s})_{s \in I, s \leq \tf (x)}$ is called the conjugate heat kernel at $x$.

\item $v_{x; \tf (x)} = \delta_x$ for all $x \in \XX$.

\item For all $s, t \in I$, $s< t$, $T \geq 0$ and any measurable function $u_s : \XX_s \to [0,1]$ with the property that if $T > 0$, 
then $u_s = \Phi \circ f_s$ for some $T^{-1/2}$-Lipschitz function $f_s : \XX_s \to \R$ \emph{(}if $T=0$, then there is no additional assumption on $u_s$\emph{)}, the following is true. 
The function
\begin{equation*}
u_t :\XX_t \longrightarrow \R, \qquad x \longmapsto \int_{\XX_s} u_s \, dv_{x;s} 
\end{equation*}
is either constant or of the form $u_t = \Phi \circ f_t$, where $f_t : \XX_t \to \R$ is $(t-s+T)^{-1/2}$-Lipschitz. Here, $\Phi$ is given by
\begin{equation*}
\Phi(x) = \int_{-\infty}^x (4\pi)^{-1/2} e^{-t^2/4} \,dt.
\end{equation*}

\item For any $t_1,t_2,t_3 \in I$, $t_1 \leq t_2 \leq t_3$, $x \in \XX_{t_3}$ we have the reproduction formula
\[ v_{x; t_1} = \int_{\XX_{t_2}} v_{\cdot; t_1} dv_{x; t_2}, \]
meaning that for any Borel set $S\subset \XX_{t_1}$
\[ v_{x;t_1} (S) = \int_{\XX_{t_2}} v_{y ; t_1} (S) dv_{x; t_2}(y). \]
\end{enumerate}
\end{defn}

Given a metric flow $\XX$ over $I$, we recall the following definition of the conjugate heat flow from \cite[Definition 3.13]{Bam20b}.

\begin{defn}[Conjugate heat flow] 
A family of probability measures $(\mu_t \in \mathcal{P} (\XX_t))_{t \in I'}$ over $I' \subset I$ is called a conjugate heat flow if for all $s, t \in I'$, $s \leq t$ we have
\begin{equation*}
\mu_s = \int_{\XX_t} \nu_{x;s} \, d\mu_t (x). 
\end{equation*}
\end{defn}

In particular, the conjugate heat kernel $(v_{x;s})_{s \in I, s \leq \tf (x)}$ is a conjugate heat flow over $I'=I \cap (-\infty, \tf(x))$.

\begin{defn}[$H$-Concentration] 
Given a constant $H>0$, a metric flow $\XX$ is called $H$-concentrated if for any $s \leq t$, $s,t \in I$, $x_1, x_2 \in \XX_t$
\begin{equation*}
\emph{\Var} (v_{x_1; s}, v_{x_2; s} ) \leq d^2_t (x_1, x_2) + H (t-s),
\end{equation*}
where the variance between two probability measures $\mu_1,\mu_2$ on $(\XX_s,d_s)$ is defined by
\begin{align*}
\emph{\text{Var}}(\mu_1,\mu_2):=\iint d_s^2(x,y) \,d\mu_1(x)d\mu_2(y).
\end{align*}
\end{defn}

For simplicity, from now on, we only consider metric flows $\XX$ over the fixed interval $I:=(-\infty,0]$. We recall the following definition from \cite[Definition 5.1]{Bam20b}.

\begin{defn}[Metric flow pair]
A pair $(\XX, (\mu_t)_{t \in I'})$ is called a metric flow pair over $I \subset \R$ if:
\begin{enumerate}
\item $I' \subset I$ with $|I \setminus I'| = 0$.
\item $\XX$ is a metric flow over $I'$.
\item $(\mu_t)_{t \in I'}$ is a conjugate heat flow on $\XX$ with $\emph{supp}\,\mu_t = \XX_t$ for all $t \in I'$.
\end{enumerate}
If $J \subset I'$, then we say that $(\XX, (\mu_t)_{t \in I})$ is fully defined over $J$. 
\end{defn}

\begin{exmp} \label{exmp:ricci}
Any pointed Ricci flow $(M^n,g(t),x_0)_{t \in I}$ induced by a Ricci shrinker corresponds to a metric flow pair $(\XX, (\mu_t)_{t \in I})$ defined as
\begin{align*}
\lc \XX:=M \times (I\setminus \{0\}) \sqcup x_0 \times \{0\}, \t:=\text{proj}_{I}, (d_t)_{t \in I}, (v_{x,t;s})_{(x,t) \in M \times I, s\in I, s\le t}, \mu_t:=v_{x_0,0;t}\rc,
\end{align*}
where $v_{x,t;s}$ denotes the conjugate heat kernel measure. From \emph{\cite[Proposition 6.5]{LW23}}, $(\XX,(\mu_t)_{t \in I})$ is an $H_n$-concentrated metric flow pair.
\end{exmp}

Next, we recall the definition of a correspondence between metric flows; see \cite[Definition 5.4]{Bam20b}.

\begin{defn}[Correspondence]
Let $(\XX^i, (\mu^i_t)_{t \in I^{\prime,i}})$ be metric flows over $I$, indexed by some $i \in \mathcal{I}$.
A correspondence between these metric flows over $I''$ is a pair of the form
\begin{equation*}
\CF := \big( (Z_t, d^Z_t)_{t \in I''},(\varphi^i_t)_{t \in I^{\prime\prime,i}, i \in \mathcal{I}} \big), 
\end{equation*}
where:
\begin{enumerate}
\item $(Z_t, d^Z_t)$ is a metric space for any $t \in I''$.
\item $I^{\prime\prime,i} \subset I'' \cap I^{\prime,i}$ for any $i \in \mathcal{I}$.
\item $\varphi^i_t : (\XX^i_t, d^i_t) \to (Z_t, d^Z_t)$ is an isometric embedding for any $i \in \mathcal{I}$ and $t \in I^{\prime\prime,i}$.
\end{enumerate}
If $J \subset I^{\prime\prime,i}$ for all $i \in \II$, we say that $\CF$ is fully defined over $J$.
\end{defn}

Given a correspondence, one can define the $\mathbb F$-distance, see \cite[Definition 5.5, Definition 5.7]{Bam20b}.

\begin{defn}[$\IF$-distance within correspondence] \label{FDWC}
We define the $\IF$-distance between two metric flow pairs within $\CF$ (uniform over $J$),
\[ d_{\IF}^{\,\CF, J} \big( (\XX^1, (\mu^1_t)_{t \in I^{\prime,1}}), (\XX^2, (\mu^2_t)_{t \in I^{\prime,2}}) \big), \] 
to be the infimum over all $r > 0$ with the property that there is a measurable subset $E \subset I''$ with
\[ J \subset I'' \setminus E \subset I^{\prime\prime,1} \cap I^{\prime\prime,2} \]
and a family of couplings $(q_t)_{t \in I'' \setminus E}$ between $\mu^1_t, \mu^2_t$ such that:
\begin{enumerate}[label=(\arabic*)]
\item $|E| \leq r^2$.
\item For all $s, t \in I'' \setminus E$, $s \leq t$, we have
\[ \int_{\XX^1_t \times \XX_t^2} d_{W_1}^{Z_s} ( (\varphi^1_s)_* \nu^1_{x^1; s}, (\varphi^2_s)_* \nu^2_{x^2; s} ) dq_t (x^1, x^2) \leq r. \]
\end{enumerate}
Here, $d_{W_1}^{Z_s}$ denotes the $W_1$-Wasserstein distance with respect to $d_s^Z$; see \cite[Page 1128]{Bam20b} for definition.
\end{defn}

\begin{defn}[$\IF$-distance] 
The $\IF$-distance between two metric flow pairs (uniform over $J$),
\[ d_{\IF}^{ J} \big( (\XX^1, (\mu^1_t)_{t \in I^{\prime,1}}), (\XX^2, (\mu^2_t)_{t \in I^{\prime,2}}) \big), \] 
is defined as the infimum of
\[ d_{\IF}^{\,\CF, J} \big( (\XX^1, (\mu^1_t)_{t \in I^{\prime,1}}), (\XX^2, (\mu^2_t)_{t \in I^{\prime,2}}) \big), \] 
over all correspondences $\CF$ between $\XX^1, \XX^2$ over $I''$ that are fully defined over $J$.
\end{defn}

It can be proved, see \cite[Theorem 6.6]{Bam20b}, that $\IF$-convergence implies $\IF$-convergence within a correspondence. More precisely,

\begin{thm} 
Let $(\XX^i, (\mu^i_t)_{t \in I^{\prime,i}})$, $i \in \IN \cup \{ \infty \}$, be metric flow pairs over $I $ that are fully defined over some $J \subset I$.
Suppose that for any compact subinterval $I_0 \subset I$
\[ d_{\IF}^{J \cap I_0} \big( (\XX^i , (\mu^i_t)_{t \in I_0 \cap I^{\prime,i}}), (\XX^\infty , (\mu^\infty_t)_{t \in I_0\cap I^{\prime,\infty}}) \big) \to 0. \] 
Then there is a correspondence $\CF$ between the metric flows $\XX^i$, $i \in \IN \cup \{ \infty \}$, over $I$ such that
\begin{equation*} 
(\XX^i, (\mu^i_t)_{t \in I^{\prime,i}}) \xrightarrow[i \to \infty]{\quad \IF, \CF, J \quad} (\XX^\infty, (\mu^\infty_t)_{t \in I^{\prime,\infty}})
\end{equation*}
on compact time intervals, in the sense that
\begin{equation*} 
d_{\IF}^{\,\CF, J \cap I_0} \big( (\XX^i, (\mu^i_t)_{t \in I_0 \cap I^{\prime,i}}), (\XX^\infty, (\mu^\infty_t)_{t \in I_0\cap I^{\prime,\infty}}) \big) \to 0
\end{equation*}
for any compact subinterval $I_0 \subset I$.
\end{thm}

Now, we can state the $\IF$-compactness theorem on Ricci flows induced by Ricci shrinkers from \cite[Theorem 6.10]{LW23}.

\begin{thm} \label{Fcom}
Let $(M_i^n,g_i(t),x_i)_{t \in I}$ be a sequence of pointed Ricci flows induced by Ricci shrinkers with the corresponding metric flow pairs $(\XX^i,(\mu_t^i)_{t \in I})$ as in Example \ref{exmp:ricci}.

After passing to a subsequence, there exists an $H_n$-concentrated metric flow pair $(\XX^\infty, (\mu^\infty_t)_{t \in I})$ for which $\XX^\infty$ is future continuous in the sense of \emph{\cite[Definition 4.7]{Bam20b}} such that the following holds.
There is a correspondence $\CF$ between the metric flows $\XX^i$, $i \in \IN \cup \{ \infty \}$, such that on compact time-intervals
\begin{equation} \label{Fconv}
(\XX^i, (\mu^i_t)_{t \in I}) \xrightarrow[i \to \infty]{\quad \IF, \CF \quad} (\XX^\infty, (\mu^\infty_t)_{t \in I}) .
\end{equation}
Moreover, the convergence (\ref{Fconv}) is uniform over any compact $J \subset I$ that only contains times at which $\XX^\infty$ is continuous, see \emph{\cite[Definition 4.7]{Bam20b}}. 
Notice that $\XX^\infty$ is continuous everywhere except possibly at a countable set of times, by \emph{\cite[Corollary 4.11]{Bam20b}}.
\end{thm}

Under the circumstance of Theorem \ref{Fcom}, we recall the following definitions from \cite[Definition 6.10]{Bam20b}.

\begin{defn}[Convergence of points] \label{Pconv}
For a sequence of times $T_i \in I$ and a sequence of points $x_i \in \XX^i_{T_i}$, $i \in \IN \cup \{ \infty \}$, we say that $x_i$ converge to $x_\infty$ within $\CF$ (and uniform over $J$) and write
\begin{equation*}
x_i \xrightarrow[i \to \infty]{\quad \CF, J \quad} x_\infty,
\end{equation*}
if $T_i \to T_\infty$, $J \subset (-\infty,T_{\infty})$, and for any compact subinterval $I_0 \subset (-\infty,T_{\infty})$, there exist measurable sets $E_i \subset I_0$, $i \in \IN $, such that
\begin{enumerate}
\item $J \cap I_0 \subset I_0 \setminus E_i \subset I^i \cap I^\infty$ for large $i$.
\item $|E_i| \to 0$.
\item $\sup_{t \in I_0 \setminus E_i} d_{W_1}^{Z_t} ( (\varphi^i_t)_* v^i_{x_i;t}, (\varphi^\infty_t)_* v^{\infty}_{x_\infty;t}) \to 0$.
\end{enumerate}

We say $x_i$ strictly converge to $x_\infty$ within $\CF$ if $T_i=T_{\infty}$, the convergence \eqref{Fconv} is uniform at $T_{\infty}$, and
\[ \varphi^i_{T_{\infty}}(x_i) \xrightarrow[i \to \infty]{\quad \quad} \varphi^\infty_{T_{\infty}}(x_\infty). \]
\end{defn}

From \cite[Theorem 6.13]{Bam20b}, the strict convergence implies convergence of points.

Next, we recall the following definition of Ricci flow spacetime from~\cite[Definition 9.1]{Bam20b}; see also~\cite[Definition 1.2]{KL17}. Here, for simplicity, we further require that each time slice of the Ricci flow spacetime is connected and of the same dimension $n$.

\begin{defn}[Ricci flow spacetime] \label{def_RF_spacetime}
An $n$-dimensional Ricci flow spacetime over an interval $I \subset \R$ is a tuple $(\MM, \mathfrak{t}, \partial_{\mathfrak{t}}, g)$ with the following properties:
\begin{enumerate}[label=(\arabic*)]
\item $\MM$ is an $(n+1)$-dimensional smooth manifold with smooth boundary $\partial \MM$, and $\MM$ is a disjoint union of smooth manifolds of dimension $n$.

\item $\mathfrak{t} : \MM \to I$ is a smooth function without critical points. For any $t \in I$ we denote by $\MM_t := \mathfrak{t}^{-1} (t) \subset \MM$ the time-$t$-slice of $\MM$.

\item $\tf (\partial \MM) \subset \partial I$.

\item $\partial_{\mathfrak{t}}$ is a smooth vector field on $\MM$ that satisfies $\partial_{\mathfrak{t}} \mathfrak{t} \equiv 1$.

\item $g$ is a smooth inner product on the spatial subbundle $\ker (d \mathfrak{t} ) \subset T \MM$.
For any $t \in I$ we denote by $g_t$ the restriction of $g$ to the time-$t$-slice $\MM_t$.

\item $(\MM_t, g_t)$ is connected for all $t \in I$.

\item $g$ satisfies the Ricci flow equation: $\mathcal{L}_{\partial_\mathfrak{t}} g = - 2 Ric (g)$.
Here $Ric (g)$ denotes the symmetric $(0,2)$-tensor on $\ker (d \mathfrak{t} )$ that restricts to the Ricci tensor of $(\MM_t, g_t)$ for all $t \in I$.
\end{enumerate}
\end{defn}

Obviously, a conventional Ricci flow $(M,g(t))_{t \in I}$ is a Ricci flow spacetime by setting $\mathcal M=M \times I$, $\tf$ to be the projection on the time factor, and $\partial_{\tf}$ to be the unit vector on $I$.

Suppose we further assume that $(M_i^n,g_i(t))$ are induced Ricci flows of a sequence of Ricci shrinkers in the moduli space $\mathcal M(A)$. In that case, one has the following structure theorem regarding the limit metric flow (see~\cite[Theorem 6.12]{LW23}), which is a direct generalization of the results of \cite{Bam20c}.

\begin{thm} \label{Thm_limit}
Let $(M_i^n,g_i(t),x_i)_{t \in I}$ be a sequence of pointed Ricci flows induced by Ricci shrinkers in $\mathcal M(A)$ and $(\XX^\infty, (\mu^\infty_t)_{t \in I})$ the limit metric flow pair obtained in Theorem~\ref{Fcom}. Then the following properties hold.

\begin{enumerate}[label=(\arabic*)]
\item There exists a decomposition
\begin{equation*}
\XX^{\infty}_{0}=\{x_{\infty}\},\quad \XX^{\infty}_{t<0}=\mathcal R \sqcup \mathcal S,
\end{equation*}
such that $\mathcal R$ is given by an $n$-dimensional Ricci flow spacetime $(\mathcal R, \tf, \partial^{\infty}_{\tf}, g^{\infty})$, in the sense of \emph{\cite[Definition 9.1]{Bam20b}} and $\emph{\text{dim}}_{\mathcal M^*}(\mathcal S) \le n-2$, 
where $\emph{\text{dim}}_{\mathcal M^*}$ denotes the $*$-Minkowski dimension in \emph{\cite[Definition 3.31]{Bam20b}}. Moreover, $\mu^{\infty}_t(\MS_t)=0$ for any $t<0$.

\item Every tangent flow $(\XX',(v_{x_{\infty}';t})_{t \le 0})$ at every point $x \in \XX^{\infty}$ is a metric soliton in the sense of \emph{\cite[Definition 3.42]{Bam20b}}.
 Moreover, $\XX'$ is the Gaussian soliton iff $x \in \mathcal R$. If $x \in \MS$, the singular set of $(\XX',(v_{x_{\infty}';t})_{t \le 0})$ on each $t <0$ has Minkowski dimension at most $n-4$. 
 In particular, if $n=3$, the metric soliton is a smooth Ricci flow associated with a $3$-dimensional Ricci shrinker. If $n=4$, each slice of the metric soliton is a smooth Ricci shrinker orbifold with isolated singularities.

\item $\mathcal R_t=\mathcal R \cap \XX^{\infty}_t$ is open such that the restriction of $d_t$ on $\mathcal R_t$ agrees with the length metric of $g_t$.

\item The convergence (\ref{Fconv}) is smooth on $\mathcal R$, in the following sense. There exists an increasing sequence $U_1 \subset U_2 \subset \ldots \subset \mathcal R$ of open subsets with $\bigcup_{i=1}^\infty U_i = \mathcal R$, 
open subsets $V_i \subset M_i \times I$, time-preserving diffeomorphisms $\phi_i : U_i \to V_i$ and a sequence $\ep_i \to 0$ such that the following holds:
\begin{enumerate}[label=(\alph*)]
\item We have
\begin{align*}
\Vert \phi_i^* g^i - g^\infty \Vert_{C^{[\ep_i^{-1}]} ( U_i)} & \leq \ep_i, \\
\Vert \phi_i^* \partial^i_{\tf} - \partial^\infty_{\tf} \Vert_{C^{[\ep_i^{-1}]} ( U_i)} &\leq \ep_i, \\
\Vert w^i \circ \phi_i - w^\infty \Vert_{C^{[\ep_i^{-1}]} ( U_i)} &\leq \ep_i,
\end{align*}
where $g^i$ is the spacetime metric induced by $g_i(t)$, and $w^i$ is the conjugate heat kernel defined by $d\mu^i=w^i dg^i$, $i \in \IN \cup \{ \infty \}$.

\item Let $y_\infty \in \mathcal R$ and $y_i \in M_i \times (-\infty,0)$.
Then $y_i$ converges to $y_\infty$ within $\CF$ \emph{(cf. \cite[Definition 6.10]{Bam20b})} if and only if $y_i \in V_i$ for large $i$ and $\phi_i^{-1} (y_i) \to y_\infty$ in $\mathcal R$.

\item If the convergence (\ref{Fconv}) is uniform at some time $t \in I$, then for any compact subset $K \subset \RR_t$ and for the same subsequence we have
\[ \sup_{x \in K \cap U_i} d^Z_t (\varphi^i_t (\phi_i(x)), \varphi^\infty_t (x) ) \longrightarrow 0. \]
\end{enumerate}
\end{enumerate}
\end{thm}

Next, we consider the simple case that a conventional Ricci flow gives the limit metric flow.

\begin{prop} \label{prop_limit2}
With the same assumptions as Theorem \ref{Thm_limit}, suppose $\XX^{\infty}_{t<0}$ is induced by a smooth Ricci flow $(M_{\infty}^n,g_{\infty}(t))_{t \in (-\infty,0)}$ with bounded curvature on any compact time interval. 
For any $\bar t <0$, we assume $(z_i,\bar t)$ is an $H$-center of $x_i$, for a constant $H \ge H_n$. Then, there exists a point $z_{\infty} \in M_{\infty}$ such that
\begin{equation} \label{App:EX1}
(M^n_i,g_i(t),z_i)_{t \in (-\infty,\bar t]} \longright{pointed-{C}^{\infty}-Cheeger-Gromov} (M^n_\infty, g_{\infty}(t),z_{\infty})_{t \in (-\infty,\bar t]}.
\end{equation}
\end{prop}

\begin{proof}
The Cheeger-Gromov convergence of the Ricci flows is essentially implied by Theorem \ref{Thm_limit} (4). 
To finish the proof, we only need to prove $(z_i,\bar t)$ converge to $(z_{\infty},\bar t)$ within $\CF$ (see Definition \ref{Pconv}), where $\CF$ is the correspondence obtained in Theorem \ref{Fcom}.

It is clear that $\XX^{\infty}$ is continuous at $\bar t$ in the sense of \cite[Definition 4.7]{Bam20b}, since $\XX^{\infty}_{t<0}$ is induced by a smooth Ricci flow. 
Consequently, the convergence (\ref{Fconv}) is uniform at $\bar t$ by Theorem \ref{Fcom}. Therefore, we conclude that
\begin{align} \label{App:E1}
d_{W_1}^{Z_{\bar t}} ( (\varphi^i_{\bar t})_* \mu^i_{\bar t} , (\varphi^{\infty}_{\bar t})_* \mu^{\infty}_{\bar t}) \to 0.
\end{align}
Now, we set $r=\sqrt{10H|\bar t|}$, $Y_i=\varphi^i_{\bar t} \lc B_{g_i(\bar t)}(z_i,r) \rc$ and $Y_{\infty}=\varphi^{\infty}_{\bar t} \lc B_{g_{\infty}(\bar t)}(z'_{\infty},r) \rc$, where $(z'_{\infty},\bar t)$ is an $H$-center of $x_{\infty}$. 

\textbf{Claim}: There exists a sequence $y_i \in Y_i$ which converges to $y_{\infty} \in Y_{\infty}$ in $Z_{\bar t}$.

\emph{Proof of Claim}: Otherwise, there exists an $\ep >0$ such that $d_{Z_{\bar t}}(x,y) \ge \ep$ for $(x,y) \in Y_i \times Y_{\infty}$. It follows from \eqref{App:E1} that
\begin{align} \label{App:E2}
\int_{Z_{\bar t} \times Z_{\bar t}} d_{Z_{\bar t}}(x,y) \,dq_i(x,y) \to 0,
\end{align}
where $q_i \in \mathcal P( Z_{\bar t}\times Z_{\bar t})$ is a coupling between $(\varphi^i_{\bar t})_* \mu^i_{\bar t}$ and $(\varphi^{\infty}_{\bar t})_* \mu^{\infty}_{\bar t}$. From \eqref{App:E2} and our assumption, we derive that
\begin{align} \label{App:E3}
q_i(Y_i \times Y_{\infty}) \le \ep^{-1}\int_{Y_i \times Y_{\infty}} d_{Z_{\bar t}}(x,y) \,dq_i(x,y) \to 0.
\end{align}
However, it follows from Proposition \ref{prop:conc} that 
\begin{align} \label{App:E4}
q_i((Z_{\bar t} \setminus Y_i) \times Y_{\infty}) \le q_i((Z_{\bar t} \setminus Y_i) \times Z_{\bar t})=(\varphi^{\infty}_{\bar t})_* \mu^{\infty}_{\bar t}(Z_{\bar t} \setminus Y_i) \le \frac{1}{10}
\end{align}
and for the same reason
\begin{align} \label{App:E5}
q_i( Y_i \times (Z_{\bar t} \setminus Y_{\infty})) \le \frac{1}{10}.
\end{align}
Combining \eqref{App:E3}, \eqref{App:E4} and \eqref{App:E5}, we derive a contradiction since $q_i$ is a probability measure. Therefore, the Claim is proved.

Next, we set $(z'_i,\bar t)=(\varphi^i_{\bar t})^{-1}(y_i)$ and $(z''_\infty,\bar t)=(\varphi^\infty_{\bar t})^{-1}(y_\infty)$. 
Therefore, it follows from the Claim that $(z'_i,\bar t)$ strictly converges to $(z''_{\infty},\bar t)$ within $\CF$ and hence by Theorem \ref{Thm_limit} (4)(b),
\begin{equation*}
(M^n_i,g_i(t),z'_i)_{t \in (-\infty,\bar t]} \longright{pointed-{C}^{\infty}-Cheeger-Gromov} (M^n_\infty, g_{\infty}(t),z''_{\infty})_{t \in (-\infty,\bar t]}.
\end{equation*}

Since $d_{g_i(\bar t)}(z_i,z_i') \le r$ and $d_{g_\infty (\bar t)}(z'_{\infty},z_{\infty}'') \le r$, one can easily conclude the existence of $z_{\infty}$ and \eqref{App:EX1} by using the diffeomorphisms $\phi_i$ constructed in Theorem \ref{Thm_limit} (4).
\end{proof}

We end the Appendix by proving the following version of the Hamilton-Ivey pinching estimate; see also~\cite[Corollary 2.4]{CBL07}.

\begin{prop} \label{prop_3dRFST}
Let $(\MM, \mathfrak{t}, \partial_{\mathfrak{t}}, g)$ be a $3$-dimensional Ricci flow spacetime over $I=(-\infty,0]$ such that $(\MM_t, g_t)$ is complete for any $t \in I$. Then $(\MM_t, g_t)$ has nonnegative sectional curvature for all $t \in I$.
\end{prop}
\begin{proof}
We first prove that $(\MM_t, g_t)$ has nonnegative scalar curvature for all $t \in I$.

From Definition \ref{def_RF_spacetime} (6), it is clear that $\MM$ is path-connected. For any $T>0$, we fix two points $x_0 \in \MM_0$ and $x_{-T} \in \MM_{-T}$. Then there exists a smooth curve $\gamma(t)_{t \in [-T,0]} \subset \MM$ such that 
\begin{align*}
x_t:=\gamma(t) \in \MM_t
\end{align*}
for any $t \in [-T,0]$. Since $\bigcup_{t \in [-T,0]} \overline{B_{g_t}(x_t,s)}$ is compact in $\MM$ for any $s>0$, there exists a small constant $r>0$ such that $\Rc_{g_t} \le 2r^{-2}$ on $B_{g_t}(x_t,r)$ for any $t \in [-T,0]$. 
Then it follows from the distance estimate \cite[Lemma 8.3(a)]{Pe1} that
\begin{align} \label{App:E11}
(\partial_{\tf}-\Delta_{t}) d_{g_t}(x_t,x) \ge -\frac{10}{3}r^{-1}-C_0
\end{align}
whenever $x \in \MM_t$ and $d_{g_t}(x,x_t)>r$. Here, the constant $C_0:=\sup_{t \in [-T,0]} |\gamma'(t)\vert_{\ker (d \mathfrak{t} )}|_{g_t}$. By shrinking $r$, we may assume from \eqref{App:E11} that
\begin{align} \label{App:E12}
(\partial_{\tf}-\Delta_{t}) d_{g_t}(x_t,x) \ge -4r^{-1}.
\end{align}

Now, we consider the following function 
\begin{align*}
u:= \phi \lc \frac{d_{g_t}(x_t,x)+4r^{-1}(t+T)}{Ar} \rc R_{g_t},
\end{align*}
where $\phi$ is a fixed smooth nonnegative decreasing function such that $\phi \equiv 1$ on $(-\infty,7/8]$ and $\phi \equiv 0$ on $[1,\infty)$, and $A \gg Tr^{-2}$ is a constant. 
By using \eqref{App:E12} and the same argument as in \cite[Proposition 2.1]{CBL07}, we conclude that
\begin{align} \label{App:E13}
R_{g_t}(x) \ge \min \left \{-\frac{3}{t+T},-\frac{C_1}{(Ar)^2} \right \}
\end{align}
if $d_{g_t}(x_t,x) \le 3Ar/4$ and $t \in (-T,0]$, where $C_1$ is a universal constant. Since $\MM_t$ is connected, if we let $A \to \infty$, it follows immediately from \eqref{App:E13} that
\begin{align*}
R_{g_t}(x) \ge -\frac{3}{t+T}
\end{align*}
for any $t \in (-T,0]$ and $x\in \MM_t$. By letting $T \to \infty$, we conclude that $R_{g_t} \ge 0$ on $\MM$.

Now, we set $\lambda \ge \mu \ge \nu$ to be the eigenvalues of the curvature operator. By using the same inductive argument as in \cite[Proposition 2.2]{CBL07}, we conclude that for any integer $k \ge 1$
\begin{align*} 
\lambda+\mu+k\nu \ge \min \left \{-\frac{C_k}{t+T},-\frac{C_k}{(Ar)^2} \right \}
\end{align*}
on $B_{g_t}(x_t,A/2)$, for any $t \in (-T,0]$. Here, $A \gg Tr^{-2}$ and $C_k$ is a constant depending only on $k$. For the same reason as above, it is clear that
\begin{align*}
\lambda+\mu+k\nu \ge 0
\end{align*}
on $\MM$ and hence $\nu \ge 0$. 

In sum, we have shown that $(\MM_t, g_t)$ has nonnegative sectional curvature for all $t \in I$.
\end{proof}

\vskip10pt

Yu Li, Institute of Geometry and Physics, University of Science and Technology of China, No. 96 Jinzhai Road, Hefei, Anhui Province, 230026, China; Hefei National Laboratory, No. 5099 West Wangjiang Road, Hefei, Anhui Province, 230088, China; yuli21@ustc.edu.cn.\\

Bing Wang, Institute of Geometry and Physics, School of Mathematical Sciences, University of Science and Technology of China, No. 96 Jinzhai Road, Hefei, Anhui Province, 230026, China; Hefei National Laboratory, No. 5099 West Wangjiang Road, Hefei, Anhui Province, 230088, China; topspin@ustc.edu.cn.\\

\end{document}